\documentclass[preprint,12pt]{elsarticle}

\usepackage{graphicx}
\usepackage{exscale}
\usepackage{times}
\usepackage{color}
\usepackage{amssymb}
\usepackage{setspace}
\usepackage{amsmath}

\font\pppppcarac=ptmr8y at 5pt
\font\ppppcarac=ptmr8y at 6pt
\font\pppcarac=ptmr8y at 7pt
\font\ppcarac=ptmr8y at 8pt
\font\pcarac=ptmr8y at 9pt

\newcommand{\bfA}{{\mathbf{A}}}

\newcommand{\bfH}{{\mathbf{H}}}

\newcommand{\bfQ}{{\mathbf{Q}}}

\newcommand{\bfU}{{\mathbf{U}}}
\newcommand{\bfV}{{\mathbf{V}}}
\newcommand{\bfW}{{\mathbf{W}}}
\newcommand{\bfX}{{\mathbf{X}}}
\newcommand{\bfY}{{\mathbf{Y}}}
\newcommand{\bfZ}{{\mathbf{Z}}}

\newcommand{\bfun}{{\boldsymbol{1}}}

\newcommand{\bfa}{{\mathbf{a}}}

\newcommand{\bff}{{\mathbf{f}}}
\newcommand{\bfg}{{\mathbf{g}}}

\newcommand{\bfj}{{\mathbf{j}}}

\newcommand{\bfm}{{\mathbf{m}}}

\newcommand{\bfq}{{\mathbf{q}}}

\newcommand{\bfu}{{\mathbf{u}}}
\newcommand{\bfv}{{\mathbf{v}}}
\newcommand{\bfw}{{\mathbf{w}}}
\newcommand{\bfx}{{\mathbf{x}}}
\newcommand{\bfy}{{\mathbf{y}}}
\newcommand{\bfz}{{\mathbf{z}}}

\newcommand{\bfeta}{{\boldsymbol{\eta}}}

\newcommand{\bfvarphi}{{\boldsymbol{\varphi}}}
\newcommand{\bfpsi}{{\boldsymbol{\psi}}}

\newcommand{\JJ}{{\mathbb{J}}}

\newcommand{\MM}{{\mathbb{M}}}
\newcommand{\NN}{{\mathbb{N}}}
\newcommand{\PP}{{\mathbb{P}}}
\newcommand{\RR}{{\mathbb{R}}}

\newcommand{\curC}{{\mathcal{C}}}
\newcommand{\curE}{{\mathcal{E}}}

\newcommand{\curH}{{\mathcal{H}}}
\newcommand{\curI}{{\mathcal{I}}}
\newcommand{\curJ}{{\mathcal{J}}}

\newcommand{\curL}{{\mathcal{L}}}
\newcommand{\curM}{{\mathcal{M}}}

\newcommand{\curP}{{\mathcal{P}}}

\newcommand{\curT}{{\mathcal{T}}}

\newcommand{\curV}{{\mathcal{V}}}

\newcommand{\bfcurY}{{\boldsymbol{\mathcal{Y}}}}
\newcommand{\bfcurZ}{{\boldsymbol{\mathcal{Z}}}}
\newcommand{\ar}{{\hbox{{\ppppcarac ar}}}}
\newcommand{\ad}{{\hbox{{\ppcarac ad}}}}
\newcommand{\pad}{{\hbox{{\ppppcarac ad}}}}

\newcommand{\PCA}{{{\hbox{{\pppppcarac PCA}}}}}
\newcommand{\err}{\hbox{{\pcarac err}}}

\newcommand{\tr}{\hbox{{\pcarac Tr}}\,}
\newcommand{\simul}{{\hbox{{\ppppcarac sim}}}}

\newcommand{\DM}{{\hbox{{\pppppcarac DM}}}}
\newcommand{\opt}{{\hbox{{\pppcarac opt}}}}
\newcommand{\popt}{{\hbox{{\ppppcarac opt}}}}
\newcommand{\ppopt}{{\hbox{{\pppppcarac opt}}}}
\newcommand{\MC}{{\hbox{{\ppppcarac MC}}}}
\newcommand{\pMC}{{\hbox{{\pppppcarac MC}}}}
\newcommand{\st}{{\hbox{{\pppcarac st}}}}

\newcommand{\pmax}{{\hbox{{\pppcarac max}}}}
\newcommand{\psup}{{\hbox{{\pppcarac sup}}}}
\newcommand{\papprox}{{\hbox{{\pppcarac app}}}}

\newtheorem{theorem}{Theorem}
\newtheorem{definition}{Definition}
\newtheorem{lemma}{Lemma}
\newtheorem{proposition}{Proposition}
\newtheorem{corollary}{Corollary}
\newproof{proof}{Proof}
\newdefinition{remark}{Remark}
\newdefinition{hypothesis}{Hypothesis}
\newdefinition{notation}{Notation}

\journal{arXiv}

\begin{document}

\begin{frontmatter}

\title{Probabilistic Learning on Manifolds}


\author[1]{Christian SOIZE \corref{cor1}}
\ead{christian.soize@univ-eiffel.fr}
\author[2]{Roger GHANEM}
\ead{ghanem@usc.edu}
\cortext[cor1]{Corresponding author: C. Soize, christian.soize@univ-eiffel.fr}
\address[1]{Universit\'e Gustave Eiffel, Laboratoire Mod\'elisation et Simulation Multi-Echelle, MSME UMR 8208, 5 bd Descartes, 77454 Marne-la-Vall\'ee, France}
\address[2]{University of Southern California, Viterbi School of Engineering, 210 KAP Hall, Los Angeles, CA 90089, United States}

\begin{abstract}
 This paper presents mathematical results in support of the methodology of the probabilistic learning on manifolds (PLoM) recently introduced by the authors, which has been used with success for analyzing complex engineering systems.  The PLoM considers a given initial dataset constituted of a small number of points given in an Euclidean space, which are interpreted as independent realizations of a vector-valued random variable for which its non-Gaussian probability measure is unknown but is, \textit{a priori}, concentrated in an unknown  subset of the Euclidean space. The objective is to construct a learned dataset constituted of additional realizations that allow the evaluation of converged statistics.  A transport of the probability measure estimated with the initial dataset is done through a linear transformation  constructed using a reduced-order diffusion-maps basis. In this paper, it is proven that this transported measure is a marginal distribution of the invariant measure of a reduced-order It\^o stochastic differential equation that corresponds to a dissipative Hamiltonian dynamical system. This construction allows for preserving the concentration of the probability measure.  This property is shown by analyzing a distance between the random matrix constructed with the PLoM and the matrix representing the initial dataset, as a function of the dimension of the basis. It is further proven that this distance has a minimum for a dimension of the reduced-order diffusion-maps basis that is strictly smaller than the number of points in the initial dataset.   Finally, a brief numerical application illustrates the mathematical results.
\end{abstract}


%
\end{frontmatter}
\section{Introduction}
\label{sec:1}
In this paper, mathematical results are presented for justifying and clarifying the methodology of  \textit{probabilistic learning on manifolds} (PLoM), initially  introduced in \cite{Soize2016}, completed in \cite{Soize2019b}, extended to the case of the polynomial chaos representation  \cite{Soize2017a}, to the case of  the learning in presence of physics constraints \cite{Soize2019c}, and  to the sampling of Bayesian posteriors \cite{Soize2019d}.
The PLoM has been used with success for the probabilistic nonconvex optimization under constraints and  uncertainties \cite{Ghanem2018a} that has allowed for analyzing very complex engineering systems such as the hypersonic combustion flows \cite{Ghanem2019,Soize2019}, the optimal placement of wells \cite{Ghanem2018b},  the design optimization under uncertainties of mesoscale implants \cite{Soize2018a}, the quantification of model uncertainties in nonlinear computational dynamics \cite{Soize2019a}, the fracture paths in random composites \cite{Guilleminot2019}.

The proposed PLoM, which can be viewed either as a supervised or an unsupervised machine learning, considers a given initial dataset constituted of $N$ given points $\bfeta_d^1,\ldots ,$ $\bfeta_d^N$ in $\RR^\nu$, which are interpreted as independent realizations of a $\RR^\nu$-valued random variable $\bfH$ for which its non-Gaussian probability measure $p_\bfH(\bfeta)\, d\bfeta$ on $\RR^\nu$ is unknown but is, \textit{a priori}, concentrated in an unknown  subset of $\RR^\nu$.
Denoting by $p_\bfH^{(N)}$ the nonparametric statistical estimation of $p_\bfH$,  the sequence of probability measures $\{p_\bfH^{(N)}(\bfeta)\, d\bfeta\}_N$ on $\RR^\nu$ is convergent to $p_\bfH(\bfeta)\, d\bfeta$ for $N\rightarrow +\infty$. In the PLoM, dimension $N$ is fixed and is presumed to be relatively small (case for which only small data are available in opposite to the big-data case). Nevertheless, it is assumed that $N$ is larger than some lower bound $N_0$ needed for
to be a sufficiently accurate estimate of $p_\bfH(\bfeta)\, d\bfeta$.  Let us now define the random vector $\bfH^{(N)}$ such that its probability measure is $p_\bfH^{(N)}(\bfeta)\, d\bfeta$. We  define the random matrix $[\bfH^N] = [\bfH^1\ldots\bfH^N]$ with values in $\MM_{\nu,N}$, whose columns $\bfH^1,\ldots, \bfH^N$ are $N$ independent copies of $\bfH^{(N)}$. The matrix $[\eta_d] = [\bfeta_d^1 \ldots \bfeta_d^N] \in\MM_{\nu,N}$ is then interpreted as one realization of random matrix $[\bfH^N]$. A reduced-order diffusion-maps basis $[g_m]\in \MM_{N,m}$ of order $m < N$ is introduced by the authors for constructing a $\MM_{\nu,N}$-valued reduced-order representation $[\bfH_m^N] = [\bfZ_m]\, [g_m]$ of random matrix $[\bfH^N]$. A MCMC generator of the random matrix  $[\bfZ_m]$ with values $\MM_{\nu,m}$ is explicitly constructed as a reduced-order  Itô stochastic differential equation (ISDE) associated with a dissipative Hamiltonian dynamical system. We then consider the family $\{p_{[\bfH_m^N]}([\eta])\, d[\eta]\}_{1\leq m \leq N}$ of probability measures on $\MM_{\nu,N}$ for which the reduced-order ISDE is the MCMC generator. We prove that there exists an optimal value $m_\opt < N$ such that the probability measure
$p_{[\bfH_{m_\ppopt}^N]}([\eta])\, d[\eta]$ allows for generating an arbitrary number $n_\MC \gg N$ of independent realizations of $[\bfH_{m_\ppopt}^N]$ (the learned dataset) in preserving the concentration of the measure. This property is shown by analyzing the function $m\mapsto d_N^2(m) = E\{\Vert [\bfH_m^N] - [\eta_d]\Vert^2\} /E\{\Vert \eta_d\Vert^2\}$, which is minimum for $m=m_\opt$ and such that  $d_N^2(m_\opt) \ll d_N^2(N)$. It should be noted that, for $m=N$,
the value $d_N^2(N)$ represents the distance of the random matrix $[\bfH^N_N]$ to the initial dataset $[\eta_d]$, for which the learned dataset would be generated without using the PLoM and consequently, would involve a scattering of the generated realizations corresponding to a loss of concentration.

In the formulation proposed and analyzed, $N$ is fixed. There is  \textit{a priori} no sense in studying the convergence of the distance for $N$ going towards infinity, on the one hand because $N$ has a limited value that is supposed to be rather small and on the other hand because $N$ also represents the number of columns of the random matrix $[H^N]$. However, for a fixed small value of $N$, the convergence of the probabilistic learning with respect to $N$ can be considered by introducing an ordered subset of integers $N_1 < N_2 < \ldots < N_{i_\pmax}$ with $N_1 > 1$ and $N_{i_\pmax} = N$, and by studying the convergence as a function of $N_i$ for $i=1,\ldots , i_\pmax$. If the convergence is reached  for $i \leq i_\pmax$, then the learning process is successful; if not, this means that the value of $N$ is too small and has to be increased, that is to say, by increasing the number of points in the initial dataset. This question is outside the scope of this paper and we refer the reader to the references given in the first paragraph of this introduction, references in which this question is dealt with.
\subsection{Framework and objective of the PLoM}
\label{sec:1.1}
Probabilistic learning is a way for improving the knowledge that one has from only a small number of expensive evaluations of a computational model in order to be able to solve, for instance, a nonconvex optimization problem under uncertainties with nonlinear constraints, for which a large number of expensive evaluations would be required. To that end, statistical and probabilistic learning methods have been extensively developed and play an increasingly important role in computational physics and engineering science. In large scale model-driven design optimization under uncertainty, and more generally, in artificial intelligence for extracting information from big data, statistical learning methods have been developed in the form of surrogate models that can  be evaluated such as, Gaussian process surrogate models, Bayesian calibration methods, active learning, and when large databases are available, neural networks, generative adversarial networks, etc.. All these methodologies allow for decreasing the numerical cost of the evaluations of expensive functions. This is particularly crucial for the evaluation of statistical estimates of large scale stochastic computational models. This is a major challenge that requires the use of suitable mathematical methods and algorithms. The probabilistic learning on manifold, which is analyzed in this paper,  is a contribution towards addressing this challenge.

In the framework of supervised machine learning, a typical problem for the use of the PLoM is the following. Let $(\bfw,\bfu)\mapsto \bff(\bfw,\bfu)$ be any measurable mapping on $\RR^{n_w}\times \RR^{n_u}$ with values in $\RR^{n_q}$ representing a computational model coming, for instance, from the discretization of a boundary value problem, in which $n_w$, $n_u$, and $n_q$ are any finite integers. Let $\bfW$ and $\bfU$ be two independent (non-Gaussian) random variables defined on a probability space $(\Theta,\curT,\curP)$ with values in $\RR^{n_w}$ and $\RR^{n_u}$, for which the probability measures $P_\bfW(d\bfw)= p_\bfW(\bfw)\, d\bfw$ and $P_\bfU(d\bfu) = p_\bfU(\bfu)\, d\bfu$ are defined by the probability density functions  $p_\bfW$ and $p_\bfU$ with respect to the Lebesgue measures $d\bfw$ and $d\bfu$ on $\RR^{n_w}$ and $\RR^{n_u}$.  Random vector $\bfW$ is made up of a part of the random parameters of the computational model, which are used for controlling the system, while random vector $\bfU$ is made up of the other part of these random parameters, which are not used for controlling the system. Let $\bfQ$ be the quantities of interest (QoI) that is a random variable defined on $(\Theta,\curT,\curP)$ with values in $\RR^{n_q}$ such that
\begin{equation} \label{eq:1.1}
\bfQ = \bff(\bfW,\bfU) \, .
\end{equation}
Let us assume that $N$ calculations have been performed with the computa\-tional model (the training) whose solution is represented by equation~\eqref{eq:1.1}, allowing $N$ independent realizations $\{\bfq^j,j=1,\ldots , N\}$ of $\bfQ$ to be computed such that
$\bfq^j = \bff(\bfw^j,\bfu^j)$, in which $\{\bfw^{j} , j=1,\ldots , N\}$ and $\{\bfu^{j} , j=1,\ldots , N\}$ are $N$ independent realizations of $(\bfW,\bfU)$, which have been generated using an adapted generator for $p_\bfW$ and $p_\bfU$.  We then consider the random variable $\bfX$ with values in $\RR^n$, such that
\begin{equation} \label{eq:1.2}
\bfX = (\bfQ,\bfW) \quad , \quad n= n_q+n_w \, .
\end{equation}
The probabilistic learning is performed for $\bfX$.
The initial dataset $D_N$ related to random vector $\bfX$ is then made up of the $N$ independent realizations
$\{\bfx^{j} , j=1,\ldots , N\}$ in which $\bfx^j = (\bfq^j,\bfw^j) \in \RR^n$.
In this paper, it is assumed that the measurable mapping $\bff$ is such that the non-Gaussian probability measure $P_\bfX(d\bfx)$ of $\bfX = (\bfQ,\bfW)$ admits a density $p_\bfX(\bfx)$ with respect to the Lebesgue measure $d\bfx$ on $\RR^n$.
The probability measure of $\bfX$ is unknown and is assumed to be concentrated in a subset of $\RR^n$ that is also unknown (this concentration property is due to equations \eqref{eq:1.1} and \eqref{eq:1.2}). The objective of the PLoM proposed in \cite{Soize2016} is to construct a probabilistic model of non-Gaussian random vector $\bfX$
using only initial dataset $D_N$, which allows for generating $\nu_\simul \gg N$ additional independent realizations $\{\bfx_\ar^1,\ldots,\bfx_\ar^{\nu_\simul}\}$ in $\RR^n$ of random vector $\bfX$, preserving the concentration of its probability measure and without using the computational model. It can then be deduced  $\nu_\simul$ additional realizations $\{(\bfq_\ar^\ell , \bfw_\ar^\ell) , \ell=1,\ldots , \nu_\simul\}$  that are such that
$(\bfq_\ar^\ell , \bfw_\ar^\ell) = \bfx_\ar^\ell$. These additional realizations allow, for instance, a cost function $\JJ(\bfw) = E\{J(\bfQ,\bfW)\vert\bfW=\bfw\}$ to be evaluated, in which $(\bfq,\bfw)\mapsto J(\bfq,\bfw)$ is a given measurable real-valued mapping on $\RR^{n_q}\times\RR^{n_w}$ as well as constraints related to a nonconvex optimization problem \cite{Ghanem2018a,Ghanem2018b,Soize2018a,Ghanem2019} and this, without calling the computational model.
The objective of the paper is to present mathematical developments of the  PLoM presented in \cite{Soize2016}, which allow for justifying the methodology and
providing foundation for further developments. The main steps of the PLoM methodology are not summarized but will be recalled as we will establish the mathematical properties.
\subsection{Organization of the paper and what are the main results}
\label{sec:1.2}
In Section~\ref{sec:2}, we introduce the $\RR^\nu$-valued random variable $\bfH$ resulting from the principal component analysis (PCA) of the $\RR^n$-valued random variable $\bfX$ with $\nu\leq n$.
Section~\ref{sec:3} is devoted to the nonparametric statistical estimate $p_\bfH^{(N)}$ of the pdf $p_\bfH$ of $\bfH$ and we give Theorem~\ref{theorem:3.1} concerning the consistency of the sequence   of estimators of $p_\bfH(\bfeta)$ for all $\bfeta$ fixed in $\RR^\nu$.
Section~\ref{sec:4} deal with the definition of random matrix $[\bfH^N]$ and  Proposition~\ref{proposition:4.3} gives an explicit expression of the pdf $p_{[\bfH^N]}$ of random matrix $[\bfH^N]$.
In Section~\ref{sec:5}, we present the construction of the reduced-order diffusion-maps basis $[g_m]$ that is used by the PLoM method and we introduce the estimation of the optimal values $\varepsilon_\opt$ and $m_\opt$ of the hyperparameter $\varepsilon_\DM$ and of the reduced order $m$. In particular, we compare the hyperparameter $\varepsilon_\DM$ and the modified Silverman bandwidth $\widehat s$; we conclude that the invariant probability measure $p^{\varepsilon_\DM}(i)$ of the Markov chain allowing the diffusion-maps basis to be constructed is different from the probability measure $p_\bfH^{(N)}(\bfeta)\, d\bfeta$ of random vector $\bfH^{(N)}$ that is considered by the PLoM.
Section~\ref{sec:6} is devoted to the construction of the probability measure and its generator related to the probabilistic learning on manifolds. We introduce the reduced-order representation $[\bfH^N_m] = [\bfZ_m]\, [g_m]^T$ of random matrix $[\bfH^N]$. In a first central Theorem~\ref{theorem:6.6}, we prove that the transported probability measure $p_{[\bfZ_m]}([z])\, d[z]$ of random matrix $[\bfZ_m]$ is the marginal distribution of the invariant measure of the reduced-order ISDE that is used as the MCMC generator of random matrix $[\bfZ_m]$. We also prove in Proposition~\ref{proposition:6.7} that $p_{[\bfZ_m]}$ has a "Gaussian representation", which is a linear combination of $N^N$ products of $\nu$ Gaussian pdf on $\RR^N$. Consequently,  the use of Theorem~\ref{theorem:6.6} effectively allows realizations of random matrix $[\bfZ_m]$ to be generated, while a Gaussian generator that would be based on the Gaussian representation is unthinkable for $N > 10$, for instance.
Section~\ref{sec:7} deals with the square of the relative distance $d_N^2(m)$ of random matrix $[\bfH_m^N]$ to matrix $[\eta_d]$ of the initial dataset. This distance allows for quantifying, as a function of $m$,  the concentration of the measure $p_{[\bfH^N_m]} ([\eta])\, d[\eta]$ in the subset of $\MM_{\nu,N}$ where the initial dataset (represented by $[\eta_d]$) is located. We show that the usual MCMC generator of random matrix $[\bfH^N]$ corresponding to $m=N$, yields $d_N^2(N) \simeq 2$ (see Lemma~\ref{lemma:7.2}), and induces a loss of concentration of the probability measure. Under a "reasonable hypothesis", the second central Theorem~\ref{theorem:7.13} proves that $d_N^2(m_\opt) \ll d_N^2(N)$ in which  $m_\opt < N$ is the optimal value of $m$. This result demonstrates that the PLoM method is a better method than the usual one because it keeps the concentration of the measure.
In Section~\ref{sec:8}, we present  a justification of the  hypothesis introduced in Theorem~\ref{theorem:7.13}, based on the use of the maximum entropy principle from Information Theory.
Section~\ref{sec:9} is devoted to a brief numerical application that illustrates the mathematical results.
The conclusions follow in Section~\ref{sec:10}.
\section*{Notations}
The following notations are used:\\
$x$: lower-case Latin of Greek letters are deterministic real variables.\\
$\bfx$: boldface lower-case Latin of Greek letters are deterministic vectors.\\
$X$: upper-case Latin or Greek letters are real-valued random variables.\\
$\bfX$: boldface upper-case Latin or Greek letters are vector-valued random variables.\\
$[x]$: lower-case Latin of Greek letters between brackets are deterministic matrices.\\
$[\bfX]$: boldface upper-case letters between brackets are matrix-valued random variables.\\
$\NN$, $\RR$: set of all the integers $\{0,1,2,\ldots\}$, set of all the real numbers.\\
$\RR^n$: Euclidean vector space on $\RR$ of dimension $n$.\\
$\bfx = (x_1,\ldots,x_n)$: point in $\RR^n$.\\
$<\! \bfx,\bfy \!> = x_1 y_1 + \ldots + x_n y_n$: inner product in $\RR^n$.\\
$\Vert\bfx\Vert$:  norm in $\RR^n$ such that $\Vert\bfx\Vert^2 =<\! \bfx,\bfx \!>$.\\
$\MM_{n,m}$: set of all the $(n\times m)$ real matrices.\\
$\MM_n$: set of all the square $(n\times n)$ real matrices.\\
$\MM_n^{+0}$: set of all the positive symmetric $(n\times n)$ real matrices.\\
$\MM_n^+$: set of all the positive-definite symmetric $(n\times n)$ real matrices.\\
$\delta_{kk'}$: Kronecker's symbol.\\
$\delta_{0_\nu}$ and $\delta_{0_{\MM_{\nu,N}}}$: Dirac measure at the origin of $\RR^\nu$ and of $\MM_{\nu,N}$.\\
$[I_{n}]$: identity matrix in $\MM_n$.\\
$[x]^T$: transpose of matrix $[x]$.\\
$\tr \{[x]\}$: trace of the square matrix $[x]$.\\
$< [x] , [y] >_F = \tr \{[x]^T \, [y]\}$, inner product of matrices $[x]$ and $[y]$ in $\MM_{n,m}$.\\
$\Vert x \Vert$ or $\Vert\, [x]\, \Vert$: Frobenius norm of matrix  $[x]$ such that  $\Vert x \Vert^2 = <  [x] , [x]>_F$.\\
$E$: mathematical expectation.\\
\section{De-correlation and normalization of random vector $\bfX$ by PCA}
\label{sec:2}
In practice, initial dataset $D_N$ results from a scaling of the available data. Then, a principal component analysis (PCA) is carried out in order to statistically condition the scaled initial dataset $D_N$ through de-correlation and normalization. Let $\underline{\widehat\bfx}\in\RR^n$ and
$[\widehat C]\in\MM_n^{+0}$ be the classical empirical estimates of the mean vector and the covariance matrix of $\bfX$, constructed using $D_N$ (scaled).
Let $[\widehat\mu]\in\MM_\nu$ be the diagonal  matrix of the first $\nu$ eigenvalues $\widehat\mu_1\geq \widehat\mu_2\geq \ldots \geq \widehat\mu_\nu > 0$ of $[\widehat C]$ and let $[\widehat\Phi] \in\MM_{n,\nu}$ be the matrix of the associated orthonormal eigenvectors. For any $\varepsilon > 0$ fixed,  $\nu \leq n$ is chosen such that
$\err_\PCA(\nu) = 1- (\widehat\mu_1 +\ldots +\widehat\mu_\nu)/(\tr[\widehat C]) \leq \varepsilon$.
This PCA allows for representing $\bfX$ by
\begin{equation} \label{eq:2.1}
\bfX^\nu = \underline{\widehat\bfx} +[\widehat\Phi]\,[\widehat\mu]^{1/2}\, \bfH \quad , \quad
E\{\Vert \bfX-\bfX^\nu\Vert^2 \} \leq \varepsilon \, E\{ \Vert\bfX\Vert^2\}\, .
\end{equation}
Throughout this paper, it will be assumed that $\nu < N$.
The $N$ independent realizations $\{\bfeta_d^j, j=1,\ldots , N\}$ of the second-order $\RR^\nu$-valued random variable $\bfH$ defined on probability space $(\Theta,\curT,\curP)$ are such that
\begin{equation}  \label{eq:2.2}
\bfeta_d^j = [\widehat\mu]^{-1/2}\, [\widehat\Phi]^T\,(\bfx^j -\underline{\widehat\bfx})\in\RR^\nu \quad , \quad j=1,\ldots , N\, .
\end{equation}
Let $[\eta_d] = [ \bfeta_d^1 \ldots \bfeta_d^N]  \in\MM_{\nu,N}$ be the  matrix of the $N$ realizations of $\bfH$.
The empirical estimates $\bfm_N\in\RR^\nu$ and $[C_N]\in\MM_\nu^+$ of the mean vector and the covariance matrix of $\bfH$,  are such that
\begin{equation} \label{eq:2.3}
\bfm_N= \frac{1}{N} \sum_{j=1}^N \bfeta_d^j = 0_\nu \quad , \quad
[C_N] = \frac{1}{N-1} [\eta_d]\,[\eta_d]^T = [I_\nu] \, .
\end{equation}
It can be seen that the Frobenius norm $\Vert\eta_d\Vert$ of matrix $[\eta_d]\in\MM_{\nu,N}$ is such that
\begin{equation} \label{eq:2.4}
\Vert\eta_d\Vert^2 = \tr\{[\eta_d]^T\, [\eta_d]\} = \sum_{j=1}^N \Vert\bfeta_d^j\Vert^2 = \nu(N-1)\, .
\end{equation}
\section{Nonparametric estimate of the pdf of $\bfH$}
\label{sec:3}
As proposed in \cite{Soize2015,Soize2016}, the modification of the multidimensional
Gaussian kernel-density estimation method \cite{Duong2005,Duong2008,Filippone2011,Zougab2014}
is used for constructing the estimation $p_\bfH^{(N)}$ on $\RR^\nu$ of the pdf $p_\bfH$ of random vector $\bfH$, which is written,
$\forall\bfeta\in\RR^\nu$, as
\begin{equation} \label{eq:3.1}
p_\bfH^{(N)}(\bfeta) = \frac{1}{N} \sum_{j=1}^N \, \pi_{\nu,N}(\frac{\widehat s}{s} \, \bfeta_d^j -\bfeta)
\quad , \quad \pi_{\nu,N}(\bfeta)  = \frac{1}{(\sqrt{2\pi}\,\widehat s)^\nu}\, \exp\{-\frac{1}{2\widehat s^2}\Vert\bfeta\Vert^2\}\, ,
\end{equation}
\begin{equation} \label{eq:3.2}
s = \left(\frac{4}{N(\nu+2)}\right )^{{1}/{(\nu+4)}}\quad , \quad
\widehat s =   \frac{s}{\sqrt{s^2 +\frac{N -1}{N}}} \,  ,
\end{equation}
in which $s$ is the usual Silverman bandwidth (since $[C_N] = [I_\nu]$) (see for instance, \cite{Bowman1997}) and where $\widehat s$ has been introduced in order that
$\int_{\RR^\nu} \bfeta\, p_\bfH^{(N)}(\bfeta)\, d\bfeta = 0_\nu$ and that
$\int_{\RR^\nu} \bfeta\otimes \bfeta\, p_\bfH^{(N)}(\bfeta)\, d\bfeta = [I_\nu]$,
because, in the framework of the PLoM, we need to preserve the centering and the orthogonality property. Finally, for fixed $\nu$,
\begin{equation}\label{eq:3.3}
\hbox{if}\quad  N \rightarrow +\infty\, , \quad \hbox{then} \quad s\rightarrow 0\, , \quad\widehat s\rightarrow 0\, ,\quad \frac{\widehat s}{s} \rightarrow 1\, , \quad \frac{ s}{\widehat s} \rightarrow 1 \, .
\end{equation}
\begin{theorem}[\textbf{Sequence of estimators of} $p_\bfH(\bfeta)$ \cite{Parzen1962}]\label{theorem:3.1}
  Let us assume that $p_\bfH$ is continuous on $\RR^\nu$. For $\nu$ fixed and for $\bfeta$ given in $\RR^\nu$, let $\{P^{(N)}(\bfeta)\}_N$ be the sequence of estimators of $p_\bfH(\bfeta)$ for which
  $P^{(N)}(\bfeta) = \frac{1}{N} \sum_{j=1}^N \pi_{\nu,N}(\frac{\widehat s}{s}\widehat\bfH^j -\bfeta)$ is a
  positive-valued random variable where $\widehat\bfH^1,\ldots , \widehat\bfH^N$ are $N$ independent copies of $\bfH$. Thus, $\forall\bfeta\in\RR^n$,
  the mean value $\underline{P}^{(N)}(\bfeta) = E\{P^{(N)}(\bfeta)\}$ and the variance $\hbox{Var}\{P^{(N)}(\bfeta)\} = E\{(P^{(N)}(\bfeta)-\underline{P}^{(N)}(\bfeta))^2\}$ of $P^{(N)}(\bfeta)$ are such that
   $\lim_{N\rightarrow +\infty} \underline{P}^{(N)}(\bfeta) = p_\bfH(\bfeta)$
   and
   $\hbox{Var}\{P^{(N)}(\bfeta)\} \leq N^{-4/(\nu+4)}\, \beta_{\nu,N}\, \underline{P}^{(N)}(\bfeta)$,
    in which, for $\nu$ fixed and for $N \rightarrow +\infty$, the positive constant $\beta_{\nu,N}$  is such that
    $\beta_{\nu,N} \sim (2\pi)^{-\nu/2}\, ( (2+\nu)/4 )^{\nu/(\nu+4)}$.
\end{theorem}
\begin{proof}
The proof, inspired from \cite{Parzen1962}, is adapted to the modification $\widehat s \ne s$ used for defining the estimator.  Since $p_\bfH$ is assumed to be a continuous function, $\forall\bfeta\in\RR^\nu$,  $p_\bfH(\bfeta) = E\{\delta_{0_\nu}(\bfH-\bfeta)\}$. Using the second equation \eqref{eq:3.1}, for all $\widetilde\bfeta$ and $\bfeta$ in $\RR^\nu$, we have
$ \pi_{\nu,N}(\frac{\widehat s}{s} \widetilde\bfeta - \bfeta)\, d\widetilde\bfeta
     =  ( s / \widehat s )^\nu \,(\sqrt{2\pi}\, s )^{-\nu}\,
     \exp\{-\frac{1}{2s^2} \, \Vert \widetilde\bfeta - \frac{s}{\widehat s} \bfeta\Vert^2\} \, d\widetilde\bfeta$.
Using equation \eqref{eq:3.3} yields the following equality in the space of the bounded measures on $\RR^\nu$,
$\lim_{N\rightarrow +\infty} \pi_{\nu,N}(\frac{\widehat s}{s} \widetilde\bfeta - \bfeta)\, d\widetilde\bfeta
=  \delta_{0_\nu}(\widetilde\bfeta -\bfeta)$.
Since $\widehat\bfH^1,\ldots,\widehat\bfH^N$ are independent copies of $\bfH$, we have
$\underline{P}^{(N)}(\bfeta) = E\{\pi_{\nu,N}(\frac{\widehat s}{s} \bfH - \bfeta)\} =
\int_{\RR^\nu} \pi_{\nu,N}(\frac{\widehat s}{s} \widetilde\bfeta - \bfeta)\, p_\bfH( \widetilde\bfeta) \, d \widetilde\bfeta$.
Using the two last above equations yields the expression for the mean. Similarly, 
$E\{(P^{(N)}(\bfeta))^2\}= \frac{1}{N}\, E\{(\pi_{\nu,N}(\frac{\widehat s}{s} \bfH - \bfeta))^2\}
+(1-\frac{1}{N})\, (\underline{P}^{(N)}(\bfeta))^2$.
Consequently,
$\hbox{Var}\{P^{(N)}(\bfeta)\}  = \frac{1}{N}\, E\{(\pi_{\nu,N}(\frac{\widehat s}{s} \bfH - \bfeta))^2\}
                                                       -\frac{1}{N}\, (\underline{P}^{(N)}(\bfeta))^2
                               \leq \frac{1}{N}\, E\{(\pi_{\nu,N}(\frac{\widehat s}{s} \bfH $ $- \bfeta))^2\}$
$\leq \frac{1}{N}\, \left (\sup_{\widetilde\bfeta} \pi_{\nu,N}(\frac{\widehat s}{s}\widetilde\bfeta - \bfeta)\right )\,$ $E\{\pi_{\nu,N}(\frac{\widehat s}{s} \bfH - \bfeta) \}
                               =  \frac{1}{N}\,(\sqrt{2\pi}\, \widehat s )^{-\nu}\,\underline{P}^{(N)}(\bfeta)$.
From equation \eqref{eq:3.2} and the last inequality yield the expression for the variance in which
$\beta_{\nu,N} =(2\pi)^{-\nu/2}$ $\{ (2+\nu)/4 \}^{\nu/(\nu+4)} \, \{1-1/N\}^{\nu/2}\,
\{ 1 + (4/(2+\nu))^{2/(\nu+4)} N^{-2/(\nu+4)}(1-1/N)^{-1} \}^{\nu/2}$
that is the expression given in the theorem for $N$ sufficiently large.
\end{proof}
\begin{remark}[\textbf{Properties of the sequence of estimators of} $p_\bfH(\bfeta)$]\label{remark:3.1}
    Theorem~\ref{theorem:3.1} shows that estimator $P^{(N)}(\bfeta)$ is asymptotically unbiased. Since $\forall \bfeta\in\RR^\nu$,
    $E\{(P^{(N)}(\bfeta)-p_\bfH(\bfeta))^2\} =  \hbox{Var}\{P^{(N)}(\bfeta)\} + ( \underline{P}^{(N)}(\bfeta)- p_\bfH(\bfeta) )^2$,
    we have $\lim_{N\rightarrow +\infty} E\{(P^{(N)}(\bfeta)-p_\bfH(\bfeta))^2\} =  0$, which shows that estimator
    $P^{(N)}(\bfeta)$ is consistent. This mean-square convergence  implies the convergence in probability.
\end{remark}
\section{Definition of the random matrix $[\bfH^N]$ and its pdf}
\label{sec:4}
The introduction of a $(\nu\times N)$ random matrix $[\bfH^N]$ will allow the initial dataset to be represented using the diffusion-maps basis.
\begin{definition}[\textbf{Matrices} ${\relax [\eta]}$, ${\relax [\eta_d]}$, ${\relax [\eta_d(\bfj)]}$, \textbf{and set} $\curJ$ ] \label{definition:4.1}
Let $[\eta]$ be any matrix in $\MM_{\nu,N}$ that is written as
\begin{equation}\label{eq:4.1}
  [\eta] = [\bfeta^1 \ldots  \bfeta^N]\in \MM_{\nu,N} \quad , \quad \bfeta^\ell = (\eta_1^\ell,\ldots, \eta_\nu^\ell)\in \RR^\nu \,\, , \,\, \ell=1,\ldots , N\, ,
\end{equation}
and let $d[\eta] = \otimes_{\ell=1}^N d\bfeta^\ell$ be the measure on $\MM_{\nu,N}$ induced by the Lebesgue measures $d\bfeta^1,\ldots , d\bfeta^N$ on $\RR^\nu$. Let $[\eta_d]\in \MM_{\nu,N}$ be the matrix constructed using the $N$ points $\bfeta^j\in \RR^\nu$ defined by equation \eqref{eq:2.2},
\begin{equation}\label{eq:4.2}
  [\eta_d] = [\bfeta_d^1 \ldots  \bfeta_d^N]\in \MM_{\nu,N} \,\, , \,\, \bfeta_d^j = (\eta_{d,1}^j,\ldots, \eta_{d,\nu}^j)\in \RR^\nu \,\, , \,\, j=1,\ldots , N\, .
\end{equation}
Let $\,\bfj = (j_1,\ldots, j_N)\in\curJ$ be the multi-index of dimension $N$ with $\curJ\! = \!\{1,2,\ldots , N\}^N$ $\subset \NN^N$.
For all $\bfj$ in $\curJ$, the matrix $[\eta_d(\bfj)]\in\MM_{\nu,N}$ is defined by
\begin{equation}\label{eq:4.3}
  [\eta_d(\bfj)]_{k\ell} = \eta_{d,k}^{j_\ell} \quad , \quad  k=1,\ldots , \nu \quad , \quad  \ell=1,\ldots , N\, .
\end{equation}
Finally, we will use the following notation, $\sum_{\bfj\in\curJ} = \sum_{j_1=1}^N\ldots \sum_{j_N=1}^N$.
\end{definition}
Note that matrix $[\eta_d]$ defined by equation \eqref{eq:4.2} has to carefully be distinguished from matrix $[\eta_d(\bfj)]$ defined
by equation \eqref{eq:4.3}. Nevertheless, it can be seen that for $\bfj_0 = (1,2,\ldots ,N)\in\curJ$, we have
$[\eta_d(\bfj_0)] = [\eta_d]$.
\begin{definition}[\textbf{Random matrix} ${\relax [\bfH^N]}$] \label{definition:4.2}
Let $\bfH^{(N)}$ be the $\RR^\nu$-valued random variable defined on $(\Theta,\curT,\curP)$ for which the pdf is $p_\bfH^{(N)}$ defined by equations \eqref{eq:3.1} and \eqref{eq:3.2}. We then define the random matrix  $[\bfH^N]$ with values in $\MM_{\nu,N}$ such that $[\bfH^N] = [\bfH^1\ldots \bfH^N]$ in which $\bfH^1, \ldots , \bfH^N$ are $N$ independent copies of $\bfH^{(N)}$. From \eqref{sec:3}, it can be seen that $E\{\bfH^{(N)}\} = 0_\nu$ and that $E\{\bfH^{(N)}\otimes \bfH^{(N)}\} =[I_\nu]$.
\end{definition}
Note that in Definition~\ref{definition:4.2}, $\bfH^1, \ldots , \bfH^N$ are not taken as $N$ independent copies of $\bfH$ whose pdf $p_\bfH$ is unknown, but are taken as $N$ independent copies of $\bfH^{(N)}$ whose pdf $p_\bfH^{(N)}$ is known.
\begin{proposition}[\textbf{Probability density function of random matrix} ${\relax [\bfH^N]}$] \label{proposition:4.3}
The probability measure of random matrix $[\bfH^N]$ with values in $\MM_{\nu,N}$ admits the following density $[\eta]\mapsto p_{[\bfH^N]}([\eta])$ on  $\MM_{\nu,N}$  with respect to $d[\eta]$,
\begin{equation}\label{eq:4.4}
p_{[\bfH^N]}([\eta])  = \prod_{\ell=1}^N    \{ \frac{1}{N} \sum_{j=1}^N \frac{1}{(\sqrt{2\pi}\,\widehat s)^\nu}\,
                       \exp\{-\frac{1}{2\widehat s^2}\Vert\frac{\widehat s}{s}\bfeta_d^j - \bfeta^\ell\Vert^2\}    \} \, .
\end{equation}
\end{proposition}
\begin{proof}
Using Definition~\ref{definition:4.2} yields, $\forall \,[\eta]\in\MM_{\nu,N}$,
$p_{[\bfH^N]}([\eta]) = \Pi_{\ell=1}^N \, p_\bfH^{(N)}(\bfeta^\ell)$ and using
 equation \eqref{eq:3.1} yields equation \eqref{eq:4.4}.
\end{proof}
\section{Construction of a reduced-order diffusion-maps basis}
\label{sec:5}
To identify the subset around which the initial data are concentrated, the PLoM method \cite{Soize2016,Soize2019b} relies on the diffusion-maps method \cite{Coifman2005,Coifman2006,Coifman2006b,Lafon2006}. We use the Gaussian kernel such that, for all $\bfeta$ and $\bfeta'$ in $\RR^\nu$,
$k_{\varepsilon_\DM}( \bfeta,\bfeta')= \exp\{-(4\,\varepsilon_\DM)^{-1} \Vert\bfeta-\bfeta'\Vert^2\}$ in which $\varepsilon_\DM >0$. The
matrices $[K]$  and  $[b]$ are defined,  for all $i$ and $j$ in $\{1,\ldots , N\}$, by
$[K]_{ij} = \exp\{-(4\,\varepsilon_\DM)^{-1} \Vert\bfeta_d^i-\bfeta_d^j\Vert^2\}$ and
$[b]_{ij} = \delta_{ij} \, b_i$ with $b_i = \sum_{j'=1}^N [K]_{ij'}$.
It is assumed that $[\eta_d]$ is such that $[K]\in\MM^+_N$. Hence, the diagonal matrix $[b]$ belongs to $\MM^+_N$. Let $\PP = [b]^{-1}[K]\in\MM_N$ be the non symmetric  matrix with positive entries such that  $\sum_{j} [\PP]_{ij} =1$ for all $i$. Matrix $[\PP]$ is the transition matrix of a Markov chain that yields the probability of transition in one step.
\subsection{Diffusion-maps basis as a non orthogonal vector basis in $\RR^N$}
\label{sec:5.1}
The eigenvalues $\lambda_1,\ldots,\lambda_N$ and the associated eigenvectors $\bfpsi^1,\ldots,\bfpsi^N$ of the right-eigen\-value problem
$[\PP]\, \bfpsi^\alpha = \lambda_\alpha\, \bfpsi^\alpha$ are such that
$ 1=\lambda_1 > \lambda_2 \geq \ldots \geq \lambda_N $
and can be computed by solving the generalized eigenvalue problem
$[K]\, \bfpsi^\alpha = \lambda_\alpha\, [b]\,\bfpsi^\alpha$ with the normalization
$<\![b]\,\bfpsi^\alpha,\bfpsi^\beta\!> =\delta_{\alpha\beta}$.
The eigenvector $\bfpsi^1$ associated with $\lambda_1=1$ is a constant vector that can be written as
$\bfpsi^1 =N^{-1/2}  \Vert\bfpsi^1\Vert\,\bfun$ with $\bfun= (1,\ldots,1)\in\RR^N$.
\begin{definition}[\textbf{Reduced-order diffusion-maps basis} ${\relax{[g_m]}}$ \textbf{of order} $m$]\label{definition:5.1}
For a given integer $\kappa \geq 0$,  the diffusion-maps basis $\{\bfg^1,\ldots,\bfg^\alpha,\ldots, \bfg^N\}$ is a vector basis of $\RR^N$ defined by $\bfg^\alpha = \lambda^\kappa_\alpha\,\bfpsi^\alpha$ such that
$<\![b]\,\bfg^\alpha,\bfg^\beta\!> =\lambda_\alpha^{2\kappa}\,\delta_{\alpha\beta}$.
For a given integer $m$ with $2 < m \leq N$, we define the reduced-order diffusion-maps basis of order $m$ as the family
$\{\bfg^1,\ldots,\bfg^m\}$ that we  represent by the matrix $[g_m] = [\bfg^1 \ldots \bfg^m]\in\MM_{N,m}$ with
$\bfg^\alpha = (g_1^\alpha,\ldots ,g_N^\alpha)$ and $[g_m]_{\ell\alpha} = g_\ell^\alpha$.
\end{definition}
Note that $\{\bfg^\alpha\}_\alpha$ is not orthogonal for the inner product $<\cdot,\cdot >$, but is orthogonal for the one defined by $(\bfu,\bfv)\mapsto <\![b]\,\bfu,\bfv\!>$ on $\RR^N\times\RR^N$. It can also be seen that the construction of the reduced-order diffusion-maps basis $[g_m]$ depends, \textit{a priori}, on three parameters: the smoothing parameter $\varepsilon_\DM$, the order $m$, and the integer $\kappa$. Nevertheless, we will see in \eqref{sec:5.4} that $\kappa$ has not a role from a theoretical point of view in the proposed method, in contrary to the one used in \cite{Coifman2005}. In the PLoM, its  role is the one of an additional scaling; its value can therefore be fixed arbitrarily (for instance, it can be set to $1$ or even to $0$; in the latter case, we have $\bfg^\alpha =\bfpsi_\alpha$). As a result, the only two parameters that will be considered will be $\varepsilon_\DM$ and $m$.
\subsection{Estimation of the optimal values $\varepsilon_\opt$ and $m_\opt$ of $\varepsilon_\DM$ and $m$}
\label{sec:5.2}
\begin{hypothesis}[\textbf{On the initial data represented by matrix} ${\relax{[\eta_d]}}$]\label{hypothesis:5.2}
For a given matrix $[\eta_d]$,  the eigenvalues $\lambda_\alpha$ depend on $\varepsilon_\DM$. It is assumed that there exist a value $\varepsilon_\opt$ of $\varepsilon_\DM$ and a value $m_\opt > 2$ of integer $m$ such that
$1=\lambda_1 > \lambda_2(\varepsilon_\opt) \geq \ldots \geq \lambda_{m_\popt}(\varepsilon_\opt) \gg \lambda_{m_\popt+1}(\varepsilon_\opt)
\geq \ldots \geq \lambda_N(\varepsilon_\opt) > 0$.
\end{hypothesis}
Under Hypothesis~\ref{hypothesis:5.2}, an algorithm associated with the given initial dataset $[\eta_d]$ has been proposed in \cite{Soize2019b} for estimating the optimal value $\varepsilon_\opt$ of $\varepsilon_\DM$ and an optimal value $m_\opt$ of order $m$.
Most of time, $\varepsilon_\opt$ and $m_\opt$ can be estimated as follows.
Let $\varepsilon_\DM\mapsto \widehat m(\varepsilon_\DM)$ be the function from $]0\, , +\infty[$ into $\NN$ such that
\begin{equation}\label{eq:5.1}
\widehat m(\varepsilon_\DM) =
\arg \min_{\alpha \, \vert \, \alpha \geq 3}\left\{ \frac{\lambda_{\alpha}(\varepsilon_\DM)}{\lambda_{2}(\varepsilon_\DM)} < 0.1\right\} \, .
\end{equation}
If function $\widehat m$ is a decreasing function of $\varepsilon_\DM$ in the broad sense (if not, the general method based on Information Theory proposed in \cite{Soize2019b} should be used), then the optimal value  $\varepsilon_\opt$ of $\varepsilon_\DM$ can be chosen as the smallest value of the integer $\widehat m(\varepsilon_\opt)$ such that
\begin{equation}\label{eq:5.2}
  \{\widehat m(\varepsilon_\opt) \! < \widehat m(\varepsilon_\DM)\, , \forall \varepsilon_\DM \in \, ]0 ,\varepsilon_\opt[ \,\}
  \,\cap \, \{\widehat m(\varepsilon_\opt)  = \widehat m(\varepsilon_\DM)\, ,
  \forall \varepsilon_\DM \in \, ]\varepsilon_\opt , 1.5\,\varepsilon_\opt[ \, \} \, .
\end{equation}
The corresponding optimal value $m_\opt$ of $m$ is then given by $m_\opt = \widehat m(\varepsilon_\opt)$ and for such an optimal choice, we have seen through numerical experiments that
$  1=\lambda_1 > \lambda_2(\varepsilon_\opt) \simeq \ldots \simeq \lambda_{m_\opt}(\varepsilon_\opt) \gg \lambda_{m_\opt+1}(\varepsilon_\opt)
\geq \ldots \geq \lambda_N(\varepsilon_\opt) > 0$.
Note that such an algorithm has been used with success for many databases in engineering sciences (see \cite{Ghanem2018a,Ghanem2018b,Soize2018a,Farhat2019,Ghanem2019,Soize2019,Soize2019a,Soize2019c,Soize2019d}).
\subsection{On the relationship between hyperparameter $\varepsilon_\DM$ and the modified Silverman bandwidth $\widehat s$}
\label{sec:5.3}
The invariant measure associated with transition matrix $[\PP]$ of the one-step Markov chain is
$p^{\varepsilon_\DM}(i) = b_i\, (\sum_{j=1}^N b_j)^{-1}$, which is such that
$\sum_{i=1}^N p(j\vert i)\,p^{\varepsilon_\DM}(i)  = p^{\varepsilon_\DM}(j)$ in which $p(j\vert i) = [\PP]_{ij}$.
Let us compare the measure $p^{\varepsilon_\DM}(i) = (\sum_{j=1}^N b_j)^{-1} \, $ $\sum_{j'=1}^N \exp\{-(4\,\varepsilon_\DM)^{-1} \Vert\bfeta_d^i-\bfeta_d^{j'}\Vert^2\}$ with  $p_\bfH^{(N)}(\bfeta) \, d\bfeta$ in which $p_\bfH^{(N)}(\bfeta)$ is defined by equations \eqref{eq:3.1} and \eqref{eq:3.2}, which is written, for $N$ sufficiently large (that is to say for $\widehat s/s \sim 1$ and $\widehat s \sim s$), as
$p_\bfH^{(N)}(\bfeta) \simeq  N^{-1}(\sqrt{2\pi} s)^{-\nu}\, \sum_{j=1}^N \exp \{-(2 \,s^2)^{-1} \Vert \bfeta - \bfeta_d^j \Vert^2 \}$.
In general, for $\nu$ sufficiently large (for instance, $\nu \sim 10$), the optimal value $\varepsilon_\opt$ defined by equations \eqref{eq:5.1} and \eqref{eq:5.2} is such that $\varepsilon_\opt \gg 1$ while, since  $\nu \leq N$, equation \eqref{eq:3.2} shows that $s^2/2 < 1$.
Therefore,  $p^{\varepsilon_\DM}(i)$ is very different from the probability measure $p_\bfH^{(N)}(\bfeta)\, d\bfeta$   that corresponds to an observation of the initial dataset from inside it, that is to say, for an observation at the smallest scale. In contrast, the probability measure $p^{\varepsilon_\DM}(i)$ is the one for which the initial dataset is observed from outside it, that is to say, for an observation at a larger scale.
\subsection{Properties of the reduced-order diffusion-maps basis}
\label{sec:5.4}
\begin{definition}[\textbf{Matrices} ${\relax{[a_m]}}$ \textbf{and} ${\relax{[G_m]}}$ ]\label{definition:5.3}
For all fixed $m$, let $[g_m]\in\MM_{N,m}$ be the matrix defined in Definition~\ref{definition:5.1}.
Since matrix $[g_m]^T\,[g_m]$ $\in \MM_{m}$ is invertible, we define the matrix $[a_m] =
[g_m]\, ([g_m]^T\,[g_m])^{-1} \in \MM_{N,m}$ and the matrix
$[G_m] = [a_m] \, [g_m]^T = [g_m]\, ([g_m]^T\,[g_m])^{-1} \, [g_m]^T \in\MM_N$.
\end{definition}
It should be noted that, as announced at the end of \eqref{sec:5.1}, matrix $[G_m]$ is independent of $\lambda_1^\kappa ,\ldots , \lambda_m^\kappa$ and thus, is independent of $\kappa$.
\begin{lemma}[\textbf{Properties of} ${\relax{[G_m]}}$]\label{lemma:5.4}
For all $m$ such that $1\leq m \leq N-1$:

\noindent (i) rank$\{[G_m]\} = m$, $\tr\{[G_m]\} = m$, $[G_m]^T = [G_m]$, and $[G_m]\in \MM_N^{+0}$.

\noindent (ii) for $m=N$, we have $[G_N] = [I_N]$.

\noindent(iii) $[G_m]^2 = [G_m]$, thus $[G_m]$ is idempotent and is a projection operator.

\noindent (iv) the eigenvalue problem $[G_m]\, \bfvarphi^\alpha = \mu_\alpha \, \bfvarphi^\alpha$ is such that $\mu_1 = \ldots = \mu_m = 1$ and $\mu_{m+1} = \ldots = \mu_N = 0$. Matrix $[G_m]$ can be written as
$[G_m] = \sum_{\alpha=1}^N \mu_\alpha \bfvarphi^\alpha\otimes \bfvarphi^\alpha = \sum_{\alpha=1}^m \bfvarphi^\alpha\otimes \bfvarphi^\alpha$
in which the eigenvectors are such that $<\! \bfvarphi^\alpha , \bfvarphi^\beta\!> = \delta_{\alpha\beta}$.

\noindent (v) $[I_N] - [G_m] \in\MM_N^{+0}$.
\end{lemma}
\begin{proof}
The proof is left to the reader.
\end{proof}
\section{Probabilistic learning on manifolds (PLoM): construction of the probability measure  and its generator}
\label{sec:6}
The three main steps of the PLoM introduced in \cite{Soize2016} are the following.
 1) Construction of a MCMC generator for random matrix $[\bfH^N]$ defined in Definition~\ref{definition:4.2}, based on a nonlinear It\^o stochastic differential equation (ISDE) that will be introduced in \eqref{sec:6.1}, for which the probability measure $p_{[\bfH^N]}([\eta])\, d[\eta]$ is a marginal probability distribution of the unique invariant measure of this ISDE.
2) Definition of a reduced representation $[\bfH^N_m] = [\bfZ_m]\, [g_m]^T$ of order $m < N$ for random matrix $[\bfH^N]$ using the reduced-order diffusion-maps basis $[g_m]$ and where $[\bfZ_m]$ is a random matrix with values in $\MM_{\nu,m}$ for which its probability measure is $p_{[\bfZ_m]}([z])\, d[z]$.
3) Construction of a reduced-order ISDE for which $p_{[\bfZ_m]}([z])\, d[z]$ is a marginal probability distribution of its unique invariant measure. We will then obtain a MCMC generator of random matrix $[\bfZ_m]$ and then of random matrix $[\bfH^N_m]$, which allows  a learned dataset  $\{[\eta_\ar^\ell],\ell=1,\ldots , n_\MC\}$ to be generated with an  arbitrary number $n_\MC$ of realizations of $[\bfH^N_{m_\popt}]$.

As already explained, the PLoM methodology has been developed for small values of $N$ (small data) for which the probability measure $p_{\bfH}^{(N)} (\bfeta) \, d\bfeta$ is not necessarily converged. Therefore additional realizations that would be generated with this measure would not provide good realizations preserving the concentration.  This is the reason why, the measure $p_{[\bfH^N]}([\eta])\, d[\eta]$ is improved by introducing the transported probability measure $p_{[\bfZ_{m_\popt}]}([z])\, d[z]$ of random matrix $[\bfZ_{m_\popt}]$. It should be noted that the additional realizations of $[\bfH^N_{m_\popt}]$ are not constructed using the projection of realizations of $[\bfH^N]$ on the subspace spanned by the reduced-order diffusion-maps basis $[g_{m_\popt}]$ (that would not be correct for a small value of $N$), but are constructed using the reduced-order ISDE associated with the transported probability measure $p_{[\bfZ_{m_\popt}]}([z])\, d[z]$ allowing additional realizations of $[\bfZ_{m_\popt}]$ to be generated and then deducing the additional realizations of
$[\bfH^N_{m_\popt}] = [\bfZ_{m_\popt}] \, [g_{m_\popt}]^T$.
\subsection{MCMC generator for random matrix $[\bfH^N]$}
\label{sec:6.1}
The PLoM method begins with the construction of a MCMC generator for random matrix $[\bfH^N]$ whose pdf $p_{[\bfH^N]}$ is given by equation \eqref{eq:4.4}.
It is based on a nonlinear ISDE, formulated for a dissipative Hamiltonian dynamical system \cite{Soize2008b,Soize2015,Soize2016} for a diffusion stochastic process $\{( [\bfU(r)],[\bfV(r)]) , r\geq 0\}$ with values in $\MM_{\nu,N}\times \MM_{\nu,N}$, which admits a unique invariant measure  for which the marginal  probability distribution with respect to $[\bfU]$ is the probability measure $p_{[\bfH^N]}([\eta])\, d[\eta]$. This MCMC generator is adapted to perform its projection on the subspace spanned by the reduced-order diffusion-maps basis and in addition, a dissipative term allows the transient part of the response to be rapidly killed. This MCMC generator belongs to the class of Hamiltonian Monte Carlo methods \cite{Neal2011,Girolami2011}, which is an MCMC algorithm \cite{Kaipio2005,Robert2005,Spall2003}.
\begin{notation} [\textbf{Matrix-valued Wiener process} ${\relax{[\bfW]}}$ \textbf{and parameter} $f_0$]\label{notation:6.1}
Let us introduce the stochastic process $\{ [\bfW(r)] = [\bfW^1(r) \ldots \bfW^N(r)],r\geq 0 \}$ defined on $(\Theta,\curT,\curP)$, with values in $\MM_{\nu,N}$, independent of random matrix $[\bfH^N]$, in which the columns $\bfW^1,\ldots,\bfW^N$ are $N$ independent copies of the normalized Wiener stochastic process $\bfW = (W_1,\ldots, W_\nu)$, defined on $(\Theta,\curT,\curP)$, indexed by $\RR^+$, with values in $\RR^\nu$, such that $\bfW(0) = 0_\nu$ a.s., $E\{\bfW(r)\} = 0_\nu$, and $E\{\bfW(r)\otimes\bfW(r')\} = \min (r,r')\, [I_\nu]$.
Let $f_0 > 0$ be a free parameter that will allow the dissipation term of the nonlinear ISDE (dissipative
Hamiltonian system)  to be controlled.
\end{notation}
\begin{theorem}[\textbf{ISDE as the MCMC generator of matrix} ${\relax{[\bfH^N]}}$] \label{theorem:6.2}
Using Notation~\ref{notation:6.1}, we consider the stochastic process $\{( [\bfU(r)],[\bfV(r)]) ,$ $r\geq 0\}$ with values in $\MM_{\nu,N}\times \MM_{\nu,N}$, which verifies the following  ISDE for $r > 0$, with the initial conditions for $r=0$,
\begin{subequations} \label{eq:6.1}
  \begin{align}
  d[\bfU(r)] & = [\bfV(r)]\, dr \, , \label{eq:6.1a}\\
  d[\bfV(r)] & =[L([\bfU(r)])]\, dr -\frac{1}{2} \, f_0\, [\bfV(r)]\, dr + \sqrt{f_0}\,\, d[\bfW(r)] \, ,\label{eq:6.1b}\\
  [\bfU(0)]  & = [\eta_d]\,\,  a.s. \quad , \quad [\bfV(0)]= [v_0] \,\, a.s. \, ,\label{eq:6.1c}
  \end{align}
  \end{subequations}
in which $[\eta_d]$ is defined by equation \eqref{eq:4.2} and  where  $[v_0]$ a given matrix in $\MM_{\nu,N}$. For $k=1,\ldots,\nu$ and $\ell=1,\ldots , N$, and for $\bfu^\ell = (u_1^\ell,\ldots ,u_\nu^\ell)$ with $u_k^\ell = [u]_{k\ell}$, the matrix $[L([u])]\in\MM_{\nu,N}$ is defined, as a function of a potential $\curV$, by
$ [L([u])]_{k\ell} = -\partial\curV(\bfu^\ell)/\partial u_k^\ell$ in which
$\curV(\bfu^\ell) = -\log\{ \frac{1}{N}\sum_{j=1}^N \exp\{-\frac{1}{2\,\widehat s^2}\,
                \Vert \frac{\widehat s}{s}\, \bfeta_d^j -\bfu^\ell \Vert^2  \} \}$.
The ISDE defined by equations \eqref{eq:6.1a} and \eqref{eq:6.1b} admits the unique invariant measure
$p_{[\bfH^N],[\bfV^N]}([\eta],[v])\, d[\eta]\otimes d[v]  =   (p_{[\bfH^N]}([\eta])\, d[\eta])\otimes (p_{[\bfV^N]}([v])\, d[v])$
on $\MM_{\nu,N}\times \MM_{\nu,N}$,
 in which $p_{[\bfV^N]}$ is the Gaussian density $[v]\mapsto (2\pi)^{-\nu N/2}$ $\exp\{-\Vert v \Vert^2/2\}$ on $\MM_{\nu,N}$ and where the pdf $p_{[\bfH^N]}([\eta])$ is defined by equation \eqref{eq:4.4}. Matrix $[v_0]$ is any realization of the Gaussian pdf $p_{[\bfV^N]}$, independent of $\{[\bfW(r)],r\geq 0\}$.
\end{theorem}
\begin{proof}
Since the columns $\bfH^1,\ldots , \bfH^N$ of random matrix $[\bfH^N]$ are independent copies of random vector $\bfH^{(N)}$ (see Definition~\ref{definition:4.2}), and since the pdf of random matrix $[\bfH^N]$ is $p_{[\bfH^N]}$ defined by equation \eqref{eq:4.4}, Theorems 4, 6, and  7 in Pages 211 to 214 of \cite{Soize1994} and the expression of the invariant measure given in Page 211 of the same reference, for which the Hamiltonian is $\curH(\bfu,\bfv) = \Vert \bfv\Vert^2 /2 + \curV(\bfu)$, prove that the invariant measure is the one given in Theorem~\ref{theorem:6.2} and is unique.
\end{proof}
\subsection{Reduced representation $[\bfH_m^N]$ of random matrix $[\bfH^N]$}
\label{sec:6.2}
\begin{definition}[\textbf{Random matrix} ${\relax{[\bfH_m^N]}}$]\label{definition:6.3}
For given $\varepsilon_\DM$, $m$, and $\kappa$, the random matrix $[\bfH_m^N]$ on $(\Theta,\curT,\curP)$, with values in $\MM_{\nu,N}$, is defined by  $[\bfH^N_m] = [\bfZ_m]\, [g_m]^T$ with  $[g_m]\in\MM_{N,m}$ defined in Definition~\ref{definition:5.1} and where $[\bfZ_m]$ is a random matrix with values in $\MM_{\nu,m}$ for which its probability measure admits a pdf $p_{[\bfZ_m]}([z])$ with respect to $d[z]$.
\end{definition}
\begin{notation}[\textbf{Random vectors} ${\relax{\widehat\bfH^k}}$ \textbf{and} ${\relax{\widehat\bfZ^k}}$]\label{notation:6.4}
For $k\in\{1,\ldots, \nu\}$, let $\widehat\bfH^k = (\widehat H_1^k,\ldots,$ $\widehat H_N^k)$ be the random vector in $\RR^N$ such that
$\widehat H_j^k = [\bfH_m^N]_{kj}$ for $j\in\{1,\ldots, N\}$ and let $\widehat\bfZ^k = (\widehat Z_1^k,\ldots, \widehat Z_m^k)$
be the  random vector in $\RR^m$ such that  $\widehat Z_\alpha^k  = [\bfZ_m]_{k\alpha}$ for $\alpha\in\{1,\ldots,m\}$. Consequently,
$\widehat\bfH^k = \sum_{\alpha=1}^m  \widehat Z_\alpha^k \, \bfg^\alpha$ in which $\bfg^\alpha$  is defined in Definition~\ref{definition:5.1}.
\end{notation}
Let $\curL^0(\Theta,\RR^N)$ be the vector space of all the random variables, defined on $(\Theta,\curT,\curP)$, with values in $\RR^N$.
It can be seen that each $\RR^N$-valued random variable $\widehat\bfH^k$ belongs to the subspace $\curL^0(\Theta,\curE_m)\subset \curL^0(\Theta,\RR^N)$ in which $\curE_m\subset\RR^N$ is the subspace of $\RR^N$ spanned by $\{\bfg^1,\ldots \bfg^m\}$.
Note that, contrarily to the PCA that is a reduction following the physical coordinates axis, the  representation constructed with the reduced-order diffusion-maps basis is a reduction following the data axis.
\begin{remark}[\textbf{Relationship between} ${\relax{[\bfH^N_N]}}$ \textbf{and} ${\relax{[\bfH^N]}}$] \label{remark:6.5}
Since $[g_N]$ is a vector basis of $\RR^N$ (see Definition~\ref{definition:5.1}), for $m=N$, the random matrix $[\bfH_N^N]$ is an independent copy of random matrix $[\bfH^N]$ introduced in Definition~\ref{definition:4.2}, in which  $[\bfH^N_N] = [\bfZ_N]\, [g_N]^T$ is a representation of $[\bfH^N]$ with $[\bfZ_N] = [\bfH^N]\, [a_N]$, where $[a_N]$ is given by Definition~\ref{definition:5.3} for $m=N$.
\end{remark}
\subsection{Explicit expression of pdf $p_{[\bfZ_m]}$ and reduced-order ISDE}
\label{sec:6.3}
\begin{theorem}[\textbf{Reduced-order ISDE and  pdf} ${\relax{p_{[\bfZ_m]}}}$]\label{theorem:6.6}
The notations introduced in Definition~\ref{definition:5.3} and in Theorem~\ref{theorem:6.2} are used.
For given $\varepsilon_\DM$, $m$, and $\kappa$, let  $\{( [\bfcurZ(r)],[\bfcurY(r)]),$ $r\geq 0\}$ be the stochastic process defined on $(\Theta,\curT,\curP)$, with values in $\MM_{\nu,m}\times \MM_{\nu,m}$, which verifies the following reduced-order ISDE for all $r > 0$, with the initial conditions for $r=0$,
\begin{subequations} \label{eq:6.2}
  \begin{align}
  d[\bfcurZ(r)] & = [\bfcurY(r)]\, dr \, , \label{eq:6.2a}\\
  d[\bfcurY(r)] & =[\curL([\bfcurZ(r)])]\, dr -\frac{1}{2} \, f_0\, [\bfcurY(r)]\, dr + \sqrt{f_0}\,\, d[\bfW(r)] \, [a_m] \, ,\label{eq:6.2b}\\
  [\bfcurZ(0) ] & = [\eta_d]\,[a_m] \,\, a.s. \quad , \quad [\bfcurY(0)]= [v_0]\,[a_m] \,\, a.s. \, ,\label{eq:6.2c}
  \end{align}
\end{subequations}
in which, $\forall \,[z]\in\MM_{\nu,m}$,  $[\curL([z])] = [L([z]\,[g_m]^T)]\, [a_m]\in\MM_{\nu,m}$.
Equations \eqref{eq:6.2a} and \eqref{eq:6.2b} admit the unique invariant measure on $\MM_{\nu,m}\times \MM_{\nu,m}$,
 \begin{equation}\label{eq:6.3}
p_{[\bfZ_m],[\bfY_m]}([z],[y])\, d[z]\otimes d[y]  =   (p_{[\bfZ_m]}([z])\, d[z])\otimes (p_{[\bfY_m]}([y])\, d[y]) \, ,
\end{equation}
in which $p_{[\bfY_m]}$ is the Gaussian density $[y]\mapsto (2\pi)^{-\nu m/2}$ $\exp\{-\Vert y \Vert^2/2\}$ on $\MM_{\nu,m}$ and where the pdf $[z]\mapsto p_{[\bfZ_m]}([z])$ on $\MM_{\nu,m}$ is written as
 \begin{equation}\label{eq:6.4}
 p_{[\bfZ_m]}([z])  = c_{\nu m} \, \prod_{\ell=1}^N    \{ \sum_{j=1}^N
       \exp\{-\frac{1}{2\widehat s^2}\Vert\frac{\widehat s}{s}\bfeta_d^j - \sum_{\alpha=1}^m \bfz^\alpha g_\ell^\alpha\Vert^2\}     \} \, .
 \end{equation}
The positive parameter $c_{\nu m}$ is the constant of normalization,  $[z] = [\bfz^1 \ldots \bfz^m] \in \MM_{\nu,m}$ with $\bfz^\alpha = (z_1^\alpha,\ldots , z_\nu^\alpha)\in\RR^\nu$ and with
$z_k^\alpha = [z]_{k\alpha}$, and  $g_\ell^\alpha$ is given by Definition~\ref{definition:5.1}.
The reduced-order ISDE with initial conditions, defined by equations \eqref{eq:6.2a} to \eqref{eq:6.2c}, has a unique stochastic solution $\{( [\bfcurZ(r)],[\bfcurY(r)]), r\geq 0\}$ that is a second-order diffusion stochastic process, which is asymptotic, for $r\rightarrow +\infty$, to a stationary and ergodic stochastic process
$\{ ([\bfcurZ_\st(r_\st)],[\bfcurY_\st(r_\st)]),$ $r_\st\ge 0\}$ for the right-shift semi-group on $\RR^+ = [0,+\infty[$. For all $r_\st$ fixed in $\RR^+$, the joint probability measure of the random matrices $[\bfcurZ_\st(r_\st)]$ and $[\bfcurY_\st(r_\st)]$ is the invariant measure defined by equation \eqref{eq:6.3} and the pdf of random matrix $[\bfcurZ_\st(r_\st)]$ is defined by equation \eqref{eq:6.4}.
Consequently, equations \eqref{eq:6.2a} to \eqref{eq:6.2c} yield a MCMC generator of random matrix $[\bfZ_m]$ and parameter $f_0$ allows for killing the transient regime induced by initial conditions, in order to reach the stationary solution more quickly.
\end{theorem}
\begin{proof}
We introduce the stochastic process $\{( [\bfcurZ(r)],[\bfcurY(r)]), r\geq 0\}$ with values in $\MM_{\nu,m}\times \MM_{\nu,m}$, such that, for all $r\geq 0$, $[\bfU(r)] = [\bfcurZ(r)] \, [g_m]^T$ and $[\bfV(r)] = [\bfcurY(r)] \, [g_m]^T$ in which $[g_m]\in\MM_{N,m}$ is given by Definition~\ref{definition:5.1} and where $\{ ( [\bfU(r)],$ $[\bfV(r)] ), r\geq 0\}$ is the stochastic process with values in $\MM_{\nu,N}\times \MM_{\nu,N}$, introduced in Theorem~\ref{theorem:6.2}. Considering this change of stochastic processes, substituting them in equations \eqref{eq:6.1a} and \eqref{eq:6.1b},  and right multiplying these two equations by matrix $[a_m]$, yield equations \eqref{eq:6.2a} and \eqref{eq:6.2b}.
The initial conditions defined by equation \eqref{eq:6.2c} are similarly obtained.

\noindent (i) Proof of equation \eqref{eq:6.4}. For $m$ fixed, since the reduced representation of random matrix $[\bfH^N]$ (for which its pdf $p_{[\bfH^N]}$ is given by equation \eqref{eq:4.4}) is defined as the random matrix $[\bfH_m^N] = [\bfZ_m]\,[g_m]^T$ (see Definition~\ref{definition:6.3}), the theorem of the image of a measure by a measurable mapping allows for deducing equation \eqref{eq:6.4} of the pdf
 $p_{[\bfZ_m]}$ of random matrix $[\bfZ_m]$ with values in $\MM_{\nu,m}$.

\noindent (ii) Proof that $p_{[\bfZ_m],[\bfY_m]}([z],[y])\, d[z]\otimes d[y]$ defined by equation \eqref{eq:6.3}, with  $p_{[\bfZ_m]}([z])$ given by equation \eqref{eq:6.4}, is the invariant measure of equations \eqref{eq:6.2a} and \eqref{eq:6.2b}. For proving that, there are several possibilities. We chose to use an algebraic-based demonstration, which allows for introducing notations that will be reused in Proposition~\ref{proposition:6.7}. For simplifying the writing, the It\^o equation \eqref{eq:6.2a}-\eqref{eq:6.2b} is rewritten as the following second-order  stochastic differential equation that has to be read as an equality of generalized stochastic processes (see for instance, Chapter XI of \cite{Kree1986}),
\begin{equation}\label{eq:6.5}
  D^2_r [\bfcurZ] + \frac{1}{2} \, f_0\, D_r [\bfcurZ] + [\curL([\bfcurZ])] =  \sqrt{f_0}\,\, D_r[\bfW] \, [a_m] \, ,
\end{equation}
in which $D_r[\bfW]$ is the generalized normalized Gaussian white process resulting from the generalized derivative with respect to $r$ of the $\MM_{\nu,N}$-valued Wiener stochastic process defined in Notation~\ref{notation:6.1}. For $k=1,\ldots,\nu$, for $\alpha=1,\ldots , m$, and for
$\ell=1,\ldots, N$, we define
$\bfz^\alpha=(z_1^\alpha,\ldots,z_\nu^\alpha)\in\RR^\nu$ and
$\widehat\bfz^k=(\widehat z_1^k,\ldots,\widehat z_m^k)\in\RR^m$
with $z_k^\alpha =\widehat z_\alpha^k = [z]_{k\alpha}$.
Similarly, we define the real functions $(\bfz^1,\ldots,\bfz^m)\mapsto \Phi(\bfz^1,\ldots ,\bfz^m)$ on $\RR^\nu\times \ldots\times \RR^\nu$ and $(\widehat\bfz^1,\ldots,\widehat\bfz^\nu)\mapsto \widehat\Phi(\widehat\bfz^1,\ldots ,\widehat\bfz^\nu)$ on $\RR^m\times \ldots\times \RR^m$, such that
\begin{equation}\label{eq:6.6}
  \Phi(\bfz^1,\ldots ,\bfz^m)  = \! \sum_{\ell=1}^N \curV(\bfu^\ell) \,\, , \,\, \bfu^\ell \! = \! \sum_{\alpha=1}^m \bfz^\alpha\, g_\ell^\alpha
  \,\, , \,\, \widehat\Phi(\widehat\bfz^1,\ldots ,\widehat\bfz^\nu)  = \Phi(\bfz^1,\ldots ,\bfz^m) \, .
\end{equation}
Equation \eqref{eq:6.5} can be rewritten as $\nu$ coupled generalized stochastic equations on $\RR^m$,
$D^2_r \widehat\bfcurZ^k + \frac{1}{2} \, f_0\, D_r \widehat\bfcurZ^k +
   ([g_m]^T\,[g_m])^{-1} \, \boldsymbol{\nabla}_{\widehat\bfcurZ^k} \widehat\Phi(\widehat\bfcurZ^1,\ldots ,\widehat\bfcurZ^\nu)
  =  \sqrt{f_0}\,[a_m]^T D_r \widehat\bfW^k$,
with $k\in\{1,\ldots ,\nu\}$ and where $\{\widehat \bfW^k\}_\alpha = [\bfW]_{k\alpha}$.  Left multiplying this last equation by the invertible matrix
$[g_m]^T\,[g_m]\in\MM_m$ yields, for $k\in\{1,\ldots,\nu\}$, the following coupled equations,
$ [g_m]^T\,[g_m]\,D^2_r \widehat\bfcurZ^k + \frac{1}{2} \, f_0\,[g_m]^T\,[g_m]\, D_r \widehat\bfcurZ^k +
  \boldsymbol{\nabla}_{\widehat\bfcurZ^k} \widehat\Phi(\widehat\bfcurZ^1, $ $\ldots ,\widehat\bfcurZ^\nu)
  =  \sqrt{f_0}\,[g_m]^T D_r \widehat\bfW^k$.
Using the mathematical results given in Chapter XIII of \cite{Soize1994}, it can be deduced that the ISDE corresponding to the previous $\nu$ coupled generalized stochastic equations admits a unique invariant measure on $(\Pi_{k=1}^\nu \RR^m)  \times (\Pi_{k=1}^\nu \RR^m)$, defined by the following density with respect to
$(\otimes_{k=1}^\nu d\widehat\bfz^k)\otimes (\otimes_{k=1}^\nu d\widehat\bfy^k)$, which is
$  p(\widehat \bfz^1,\ldots,\widehat \bfz^\nu; \widehat \bfy^1,\ldots,\widehat \bfy^\nu)
  = \widehat c_{2\,\nu m} \exp\{-\frac{1}{2}\sum_{k=1}^\nu $ $<\! [g_m]^T\,[g_m]\,\widehat\bfy^k , \widehat\bfy^k\! >
         - \widehat\Phi(\widehat\bfz^1,\ldots ,\widehat\bfz^\nu)\}$.
Consequently, the joint pdf of the $\RR^\nu$-valued random variables $\bfZ^1,\ldots ,\bfZ^m$ with respect to $\otimes_{\alpha=1}^m d\bfz^\alpha$ is given, using the third equation \eqref{eq:6.6}, by
$p_{\bfZ^1,\ldots ,\bfZ^m}(\bfz^1,\ldots ,\bfz^m) = \widehat c_{\nu m} \exp\{-\Phi(\bfz^1,\ldots ,\bfz^m)\}$ and thus, using equation \eqref{eq:6.6}, the pdf of random matrix $[\bfZ_m]$ with respect to $d[z]$ is
$p_{[\bfZ_m]}([z]) = \widehat c_{\nu m} \exp\{-\sum_{\ell=1}^N \curV(\sum_{\alpha=1}^m \bfz^\alpha g_\ell^\alpha) \}$.
Using the expression of $\curV(\bfu^\ell)$ defined in Theorem~\ref{theorem:6.2} and introducing
$c_{\nu m} = \widehat c_{\nu m} / N^N$, this pdf can be rewritten as equation \eqref{eq:6.4}.

\noindent (iii) Proof of uniqueness of an asymptotic stationary and ergodic solution of equations \eqref{eq:6.2a} to \eqref{eq:6.2c}.
The use of Theorem 9 in Page 216 of \cite{Soize1994} yields the proof that  equations \eqref{eq:6.2a} to \eqref{eq:6.2c}  has a unique solution $\{( [\bfcurZ(r)],[\bfcurY(r)]), r\geq 0\}$ that is a second-order diffusion stochastic process, which is asymptotic, for $r\rightarrow +\infty$, to a unique stationary stochastic process $([\bfcurZ_\st],[\bfcurY_\st])$ having the properties given in  Theorem~\ref{theorem:6.6}.
The ergodicity of the stationary solution is directly deduced from \cite{Doob1953} or from \cite{Khasminskii2012}.
\end{proof}
\begin{proposition}[\textbf{Explicit expression of the pdf} ${\relax{ p_{[\bfZ_m]} }}$ \textbf{of} ${\relax{ [\bfZ_m] }}$]\label{proposition:6.7}
(i) The pdf $p_{[\bfZ_m]}$ of random matrix $[\bfZ_m]$ defined by equation \eqref{eq:6.4} can be rewritten, for all $[z]$ in $\MM_{\nu,m}$, as
\begin{equation}\label{eq:6.7}
p_{[\bfZ_m]}([z]) =\sum_{\bfj\in\curJ} p_\bfj(m) \, \prod_{k=1}^\nu p_{\widehat\bfZ^k}(\widehat\bfz^k;\bfj) \,\,\, , \,\,\,\widehat\bfz^k=(\widehat z^k_1,\ldots ,\widehat z^k_m)\in\RR^m \,\,\, ,\,\,\, \widehat z^k_\alpha = [z]_{k \alpha} \, ,
\end{equation}
in which for all $\bfj$ in $\curJ$ (see Definition~\ref{definition:4.1}),
\begin{subequations} \label{eq:6.8}
  \begin{align}
   p_\bfj(m) & = \gamma_\bfj(m) \, (\sum_{\bfj'\in\curJ} \gamma_{\bfj'}(m))^{-1} \quad , \quad \sum_{\bfj\in\curJ} p_\bfj(m)= 1 \, , \label{eq:6.8a}\\
   \gamma_\bfj(m) & =\exp\{-\frac{1}{2\, s^2} <  [I_N] - [G_m]\, ,\, [M_d(\bfj)] >_F\} \, . \label{eq:6.8b}
   \end{align}
\end{subequations}
Matrix $[G_m]\in\MM_N$ (see Definition~\ref{definition:5.3}) and the matrix $[M_d(\bfj)]\in\MM_N^{+0}$ is defined by
\begin{equation}\label{eq:6.9}
 [M_d(\bfj)]  = [\eta_d(\bfj)]^T\, [\eta_d(\bfj)] \quad ,\quad
 [M_d(\bfj)]_{\ell\ell'} = <\!\bfeta_d^{j_\ell} , \bfeta_d^{j_{\ell'}}\! > \, ,
\end{equation}
in which $[\eta_d(\bfj)]\in\MM_{\nu,N}$ is defined by equation \eqref{eq:4.3}. For all $k$ in $\{1,\ldots,\nu\}$,
$p_{\widehat\bfZ^k}(\cdot;\bfj)$ is the Gaussian pdf, such that, for all $\widehat\bfz^k$ in $\RR^m$,
\begin{equation}\label{eq:6.10}
p_{\widehat\bfZ^k}(\widehat\bfz^k;\bfj)= ((2\pi)^m \det[C_m])^{-1/2}\exp\{-\frac{1}{2}
\! <\! [C_m]^{-1}(\widehat\bfz^k \!- \!{\underline{\widehat{\bfz}}}^k(\bfj))\, , \, \widehat\bfz^k\! - \!{\underline{\widehat{\bfz}}}^k(\bfj)\!>\}
\, ,
\end{equation}
in which $[C_m] = \widehat s^2([g_m]^T\, [g_m])^{-1}\in\MM_m^+$, where
${\underline{\widehat{\bfz}}}^k(\bfj)= (\widehat s / s)\, [a_m]^T\, \widehat\bfeta_d^k(\bfj)\in \RR^m$
with $[a_m]\in\MM_{N,m}$ given by Definition~\ref{definition:5.3}, and where
$\widehat\bfeta_d^k(\bfj)= ( \widehat\eta_{d,1}^k(\bfj), \ldots , \widehat\eta_{d,N}^k(\bfj))$ $\in\RR^N$ with
$\widehat\eta_{d,\ell}^k(\bfj) = \eta_{d,k}^{j_\ell} = [\eta_d(\bfj)]_{k\ell}$ (see equation \eqref{eq:4.3}).
(ii) For all $\bfj$ in $\curJ$ and for all $1\leq m \leq N-1$, we have
$a_\bfj(m)\,  {\buildrel\hbox{\ppppcarac def}\over{=}} < \![I_N] - [G_m]\, , \,  [M_d(\bfj)]\! >_F \,\,\geq 0$,
$0 \! < \! \gamma_\bfj(m)\! < \! 1$,
$0 \! < \! p_\bfj(m) \! < \!1$,
and for $m=N$,  $a_\bfj(N) = 0$, $\gamma_\bfj(N) = 1$, and $ p_\bfj(N) = 1/N^N$.
\end{proposition}
\begin{proof}
(i) Equation \eqref{eq:6.4} writes
$ p_{[\bfZ_m]}([z])  = c_{\nu m} \sum_{\bfj\in\curJ} \prod_{k=1}^\nu
       \exp\{-\frac{1}{2\widehat s^2}\Vert\frac{\widehat s}{s}\widehat\bfeta_d^k(\bfj) - [g_m]\, \widehat\bfz^k\Vert^2\}$.
On the other hand,
$\forall k \in\{1,\ldots, \nu\}$, we have
$-(2\,\widehat s^2)^{-1}\Vert (\widehat s/s)\,\widehat\bfeta_d^k(\bfj) - [g_m]\, \widehat\bfz^k\Vert^2
= -(1/2) <\! [C_m]^{-1}(\widehat\bfz^k \!- \!{\underline{\widehat{\bfz}}}^k(\bfj))\, , \, \widehat\bfz^k\! - \!{\underline{\widehat{\bfz}}}^k(\bfj)\!>
+ (2\,s^2)^{-1} (<\! [G_m]\widehat\bfeta_d^k(\bfj)\, ,$ $\widehat\bfeta_d^k(\bfj)\!> - \Vert \widehat\bfeta_d^k(\bfj)\Vert^2 )$. Combining the two previous equations allows $p_{[\bfZ_m]}([z])$ to be rewritten as
$p_{[\bfZ_m]}([z])= c_{\nu m}\sum_{\bfj\in\curJ} \gamma_\bfj(m)\prod_{k=1}^\nu
\exp\{-\frac{1}{2}
\! <\! [C_m]^{-1}(\widehat\bfz^k \!- \!{\underline{\widehat{\bfz}}}^k(\bfj))\, , \, \widehat\bfz^k\! - \!{\underline{\widehat{\bfz}}}^k(\bfj)\!>\}$,
in which $\gamma_\bfj(m) = \prod_{k=1}^\nu \exp\{-\frac{1}{2\widehat s^2}( \Vert\widehat\bfeta_d^k(\bfj)\Vert^2
- <\! [G_m]\widehat\bfeta_d^k(\bfj)\, ,$ $\widehat\bfeta_d^k(\bfj)\!>)\}$, which can, finally, be rewritten as equation \eqref{eq:6.8b} with \eqref{eq:6.9}. 
The above expression of $p_{[\bfZ_m]}([z])$ is rewritten as
$p_{[\bfZ_m]}([z])=c_{\nu m} (2\pi)^{\nu m/2} (\det[C_m])^{\nu/2}$ $\sum_{\bfj\in\curJ}\gamma_\bfj(m)\prod_{k=1}^\nu
p_{\widehat\bfZ^k} (\widehat\bfz^k ; \bfj)$. The constant $c_{\nu m}$ of normalization is calculated by
$\int_{\MM_{\nu,m}} p_{[\bfZ_m]}([z]) \, d[z] = 1$. Since $\int_{\RR^m} p_{\widehat\bfZ^k} (\widehat\bfz^k ; \bfj)
\, d \widehat\bfz^k$ $ = 1$, we obtain
$c_{\nu m} = \{(2\pi)^{\nu m/2} (\det[C_m])^{\nu/2} \sum_{\bfj\in\curJ}\gamma_\bfj(m)\}^{-1}$. Equation \eqref{eq:6.7} can then be deduced in which $p_\bfj(m)$ is given by equation \eqref{eq:6.8a}.
Finally, $[M_d(\bfj)] =\sum_{k=1}^\nu \widehat\bfeta_d^k(\bfj)\otimes \widehat\bfeta_d^k(\bfj)  = [\eta_d(\bfj)]^T\, [\eta_d(\bfj)]$, which shows  that $[M_d(\bfj)]\in \MM_N^{+0}$ because  $\nu < N$ (see \eqref{sec:2}).
(ii) From Lemma~\ref{lemma:5.4}-(v), and using equation \eqref{eq:6.9}, it can be seen that $a_\bfj(m) \geq 0$. The end of the proof is easy to do.
\end{proof}
\begin{remark} [\textbf{About the algebraic representation of pdf} ${\relax{p_{[\bfZ_m]}}}$ \textbf{and its generator}] \label{remark:6.8}
(i) Equation \eqref{eq:6.7} shows that pdf $p_{[\bfZ_m]}$ on $\MM_{\nu,m}$ is a linear combination of $N^N$ products of $\nu$ Gaussian pdf on $\RR^N$. Consequently,  the use of the reduced-order ISDE given by Theorem~\ref{theorem:6.6} effectively allows realizations of random matrix $[\bfZ_m]$ to be generated, while a Gaussian generator that would be based on the representation given by equation \eqref{eq:6.7} is unthinkable.
(ii) The generation of $n_\MC \gg 1$ independent realizations $\{ [z^\ell], \ell=1,\ldots, n_\MC\}$ of random matrix $[\bfZ_m]$ is performed by using the MCMC generator defined by Theorem~\ref{theorem:6.6} in which the  reduced-order stochastic equations \eqref{eq:6.2a} to \eqref{eq:6.2c} are solved using the St\"{o}rmer-Verlet scheme \cite{Hairer2006,Burrage2007}, which is well adapted to stochastic Hamiltonian dynamical systems and which is detailed in \cite{Soize2016}. We can then deduce the learned dataset $\{ [\eta_\ar^\ell], \ell=1,\ldots, n_\MC\}$ of random matrix $[\bfH_m^N]$ such that $[\eta_\ar^\ell] = [z^\ell]\, [g_m]^T$, with an arbitrary value of realizations.
\end{remark}
\section{$L^2$-distance of random matrix $[\bfH_m^N]$ to matrix $[\eta_d]$ of the initial dataset and its analysis}
\label{sec:7}
In this section, $N$, $\nu$, $\kappa$, and $\varepsilon_\DM = \varepsilon_\opt$ are fixed. The optimal value of $m$ associated with $\varepsilon_\opt$ is $m_\opt$ as defined in \eqref{sec:5.2}. Integer $m$ varies in $\{1,\ldots , N\}$. The measure of the concentration of the probability measure $p_{[\bfH_m^N]}([\eta])\, d[\eta]$, which is informed by the initial dataset represented by matrix $[\eta_d]$, will be analyzed as a function of $m$ by using  the square $d_N^2(m)$ of the $L^2(\Theta,\MM_{\nu,N})$-distance between random matrix $[\bfH_m^N]$ and matrix $[\eta_d]$.
\begin{definition}[\textbf{Square of the relative distance} $d_N^2(m)$ \textbf{of} ${\relax{[\bfH_m^N]}}$ \textbf{to} ${\relax{[\eta_d]}}$]\label{definition:7.1}
For $m$ fixed, the square of the relative distance  of random matrix $[\bfH_m^N]$ with values in $\MM_{\nu,N}$ to matrix $[\eta_d]\in\MM_{\nu,N}$ is defined as $d_N^2(m) = E\{\Vert [\bfH_m^N] - [\eta_d]\Vert^2\} /E\{\Vert [\eta_d]\Vert^2\}$.
\end{definition}
The following Lemma gives the value of $d_N^2(N)$, which corresponds to the value of the distance if the PLoM method is not used ($m=N$). In this case, the MCMC generator of random matrix $[\bfH_N^N]$  is  given by Theorem~\ref{theorem:6.2}.
\begin{lemma}[\textbf{Value of} $d_N^2(m)$ \textbf{for} $m=N$]\label{lemma:7.2}
For $m=N$, the random matrix $[\bfH_N^N]$, which is an independent copy of random matrix $[\bfH^N]$ (see Definition~\ref{definition:4.2} and  Remark~\ref{remark:6.5}) is such that $E\{[\bfH_N^N]\} = [0_{\nu,N}]$, $E\{\Vert\, [\bfH_N^N]\,\Vert^2\} = \nu N$, and
the value of $d_N^2(m)$  for $m=N$ is
$d_N^2(N) = 1 + N/(N-1)$.
\end{lemma}
\begin{proof}
Note that $\bfH^1,\ldots ,\bfH^N$ are independent copies of $\bfH^{(N)}$ (see Definition~\ref{definition:4.2}).
(i) $E\{[\bfH_N^N]\} = E\{[\bfH^N]\} = [E\{\bfH^{(N)}\} \ldots E\{\bfH^{(N)}\}]$, and since $E\{\bfH^{(N)}\}=0_\nu$, we have $E\{[\bfH_N^N]\} = [0_{\nu,N}]$.
(ii) $E\{\Vert\, [\bfH_N^N]\, \Vert^2\}= E\{\Vert\,[\bfH^N]\,\Vert^2\} = \sum_{j=1}^N E\{\Vert\bfH^j\Vert^2\}$
$= N\, E\{ \Vert \bfH^{(N)}\Vert^2\}$, therefore, $E\{\Vert\,\bfH^{(N)}\,\Vert^2\}$ $ = \tr\{[I_\nu]\} = \nu$, and consequently, we have $E\{\Vert \,[\bfH_N^N]\, \Vert^2\} = \nu N$.
(iii) Using Definition~\ref{definition:7.1} and equation \eqref{eq:2.4} yields
$d_N^2(N) = (\nu(N-1))^{-1}\, (\, E\{\Vert \, [\bfH_N^N] \,\Vert^2\} -2 <\! E\{[\bfH_N^N]\} , [\eta_d] \!>_F + \Vert \eta_d\Vert^2)$. The result is obtained using (i), (ii), and equation \eqref{eq:2.4}.
\end{proof}
\begin{proposition}[\textbf{Expression of} $d_N^2(m)$]\label{proposition:7.3}
Let $m$ be fixed.  We have
\begin{subequations}\label{eq:7.1}
    \begin{align}
    E\{[\bfH_m^N]\} & = \sum_{\bfj\in\curJ} p_\bfj(m)\, \frac{\widehat s}{s} [\eta_d(\bfj)]\, [G_m]\in \MM_{\nu,N}\, , \label{eq:7.1a}\\
    E\{\Vert [\bfH_m^N] \vert^2\}  & = \sum_{\bfj\in\curJ} p_\bfj(m)\,(\nu\widehat s^2 m  +\frac{\widehat s^2}{s^2}
            < \![G_m]\, , [M_d(\bfj)] \!>_F) \, ,  \label{eq:7.1b}
    \end{align}
\end{subequations}
 in which $p_\bfj(m)$ is given by equation \eqref{eq:6.8a}, $[\eta_d(\bfj)]$ is defined by equation \eqref{eq:4.3}, $[G_m]$ by Definition~\ref{definition:5.3},  and $[M_d(\bfj)]$ by equation \eqref{eq:6.9}, and we have
 \begin{subequations}\label{eq:7.2}
   \begin{align}
   d_N^2(m) &= 1 +\frac{m\widehat s^2}{N-1} + \frac{1}{\Vert\eta_d\Vert^2}
   \sum_{\bfj\in \curJ} p_\bfj(m) < \! [G_m] \, , [B_d(\bfj)]\! >_F \, , \label{eq:7.2a}\\
   [B_d(\bfj)] &= \frac{\widehat s^2}{ s^2} \,[M_d(\bfj)] - 2 \frac{\widehat s}{s}\, [\eta_d(\bfj)]^T\,[\eta_d]  \in \MM_N  \, . \label{eq:7.2b}
   \end{align}
 \end{subequations}
The entries of $[B_d(\bfj)]$ are $[B_d(\bfj)]_{\ell\ell'} = (\widehat s^2/ s^2) < \!\bfeta_d^{j_\ell} , \bfeta_d^{j_{\ell'}}\! > - (2\widehat s/ s)< \!\bfeta_d^{j_\ell} , \bfeta_d^{\ell'}\! >$.
\end{proposition}
\begin{proof}
Since $[\bfH_m^N] = [\bfZ_m]\, [g_m]^T$ (see Definition~\ref{definition:6.3}), it can be seen that
$E\{[\bfH_m^N]\}$ $ = [\curM_1^N(m)]\, [g_m]^T$ and $E\{\Vert [\bfH_m^N] \vert^2\} =  < \![\curM_2^N(m)]\, , [g_m]^T\, [g_m] \! >_F$, in which, using equation \eqref{eq:6.7} for $p_{[\bfZ_m]}$,
\begin{align*}
 [\curM_1^N(m)]&  = \sum_{\bfj\in\curJ} p_\bfj(m)
     \int_{\RR^m} \ldots \int_{\RR^m} [\widehat\bfz^1 \ldots \widehat\bfz^\nu]^T\,
       \otimes_{k=1}^\nu \{p_{\widehat\bfZ^k}(\widehat\bfz^k;\bfj) \,d\widehat\bfz^k \}\in\MM_{\nu,m} \, , \\
       [\curM_2^N(m)] & = \sum_{k'=1}^\nu \sum_{\bfj\in\curJ} p_\bfj(m)
            \prod_{k=1}^\nu \left\{ \int_{\RR^m} \widehat\bfz^{k'}\!\otimes \widehat\bfz^{k'}\,
                  p_{\widehat\bfZ^k}(\widehat\bfz^k;\bfj)\, d\widehat\bfz^k \right\}\in\MM_{\nu,m} \, .
 \end{align*}

\noindent (i) Calculation of $E\{[\bfH_m^N]\}$. Using equation \eqref{eq:6.10} and the expression of
${\underline{\widehat{\bfz}}}^k(\bfj)$ defined in Proposition~\ref{proposition:6.7} yield
$[\curM_1^N(m)]= \sum_{\bfj\in\curJ} p_\bfj(m)\, [{\underline{\widehat{\bfz}}}(\bfj)]^T$ with
$[{\underline{\widehat{\bfz}}}(\bfj)]=
[{\underline{\widehat{\bfz}}}^1(\bfj)\ldots {\underline{\widehat{\bfz}}}^\nu(\bfj) ]$ $ =
(\widehat s / s)\, [a_m]^T\, [\eta_d(\bfj)]^T$. Since $[a_m]\, [g_m]^T = [G_m]$ (see Definition~\ref{definition:5.3}), we obtain equation \eqref{eq:7.1a}.

 \noindent (ii) Calculation of $E\{\Vert [\bfH_m^N] \vert^2\}$. Equation \eqref{eq:6.10} shows that
 $\int_{\RR^m}  \widehat\bfz^{k}\!\otimes \widehat\bfz^{k} \, p_{\widehat\bfZ^k}(\widehat\bfz^k;\bfj)\, d\widehat\bfz^k$ $= [C_m] + {\underline{\widehat{\bfz}}}^k(\bfj) \otimes {\underline{\widehat{\bfz}}}^k(\bfj)$.
 Since $\int_{\RR^m} p_{\widehat\bfZ^k}(\widehat\bfz^k;\bfj)\, d\widehat\bfz^k  = 1$
 and that  $\sum_{k'=1}^\nu {\underline{\widehat{\bfz}}}^{k'}(\bfj) \otimes {\underline{\widehat{\bfz}}}^{k'}(\bfj)
 = [{\underline{\widehat{\bfz}}}(\bfj)]\, [{\underline{\widehat{\bfz}}}(\bfj)]^T$,
 we have  $[\curM_2^N(m)]= \sum_{\bfj\in\curJ} p_\bfj(m)\,( \nu\, [C_m] + [{\underline{\widehat{\bfz}}}(\bfj)]\, [{\underline{\widehat{\bfz}}}(\bfj)]^T)$.
 It can then be deduced that  $ E\{\Vert [\bfH_m^N] \vert^2\}   = \sum_{\bfj\in\curJ} p_\bfj(m)\,(\nu\, \curI_1(m) + \curI_2(m,\bfj))$.
 Using $[C_m] =\widehat s^2\, ([g_m]^T\, [g_m])^{-1}$ (see Proposition~\ref{proposition:6.7}) and the expression of $[G_m]$ given in Definition~\ref{definition:5.3} yield
 $\curI_1(m)= <\! [C_m]\, ,[g_m]^T [g_m] \!>_F = \widehat s^2\, \tr\{[G_m]\} = \widehat s^2 \, m$.
 On the other hand,
 $\curI_2(m,\bfj)= \Vert [{\underline{\widehat{\bfz}}}(\bfj)]^T \, [g_m]^T\Vert^2$
 and since $[{\underline{\widehat{\bfz}}}(\bfj)]^T = (\widehat s / s)\,  [\eta_d(\bfj)]\, [a_m]$,
 we have  $[{\underline{{\bfz}}}(\bfj)]^T\, [g_m]^T = (\widehat s / s)\,  [\eta_d(\bfj)]\, [G_m]$. Therefore,
  $\curI_2(m,\bfj)= (\widehat s^2 / s^2)\Vert [\eta_d(\bfj)]\,$ $[G_m]\Vert^2$.
 Since $[G_m]^2 = [G_m]$ (see Lemma~\ref{lemma:5.4}-(iii))
  and that $[M_d(\bfj)] = [\eta_d(\bfj)]^T\, [\eta_d(\bfj)]$ (see equation \eqref{eq:6.9}),
  we obtain  $\curI_2(m,\bfj)=  (\widehat s^2 / s^2) < \! [G_m]\, ,[M_d(\bfj)] \!>_F$. By substitution, we obtain equation \eqref{eq:7.1b}.

 \noindent (iii) Calculation of $d_N^2(m)$. From Definition~\ref{definition:7.1} of $d_N^2(m)$,  we have
 $d_N^2(m) = \Vert\eta_d\Vert^{-2}$ $ (  E\{ \Vert [\bfH_m^N ]\Vert^2\} -2 < \! E\{[\bfH_m^N]\} \, , [\eta_d] \!>_F + \Vert \eta_d\Vert^2)$. Using equations \eqref{eq:7.1a} and \eqref{eq:7.1b} allows for proving equation  \eqref{eq:7.2a} with \eqref{eq:7.2b}.
\end{proof}
In order to apply Proposition~\ref{proposition:7.3} for the case $m=N$, we need the results given in the following lemma.
\begin{lemma}\label{lemma:7.4}
Using  Definition~\ref{definition:4.1} and equations \eqref{eq:6.9} and  \eqref{eq:7.2b}, we have
\begin{subequations}\label{eq:7.3}
\begin{align}
  \frac{1}{N^N} & \, \sum_{\bfj\in\curJ}\,  [\eta_d(\bfj)] = [0_{\nu,N}]   \, , \label{eq:7.3a} \\
  \frac{1}{N^N} & \, \sum_{\bfj\in\curJ}\,  [M_d(\bfj)] = \frac{1}{N}  \Vert\eta_d\Vert^2 \, [I_N]   \,  \label{eq:7.3b} \\
  \frac{1}{N^N} & \, \sum_{\bfj\in\curJ}\,  [B_d(\bfj)] = \frac{\widehat s^2}{s^2}\frac{1}{N}  \Vert\eta_d\Vert^2 \, [I_N]  \, . \label{eq:7.3c}
\end{align}
\end{subequations}
\end{lemma}
\begin{proof}
(i) Proof of equation \eqref{eq:7.3a}: for all fixed $\ell$, we have $N^{-N}\sum_{\bfj\in\curJ}  \bfeta_d^{j_\ell} =   N^{-N}\sum_{j_1}^N \ldots \sum_{j_N}^N  \bfeta_d^{j_\ell} = N^{-1}\sum_{j_\ell=1}^N  \bfeta_d^{j_\ell}=  0_\nu$, taking into account equation \eqref{eq:2.3}.
(ii) Proof of equation \eqref{eq:7.3b}: for $\ell=\ell'$, it can be seen that $N^{-N}\sum_{\bfj\in\curJ} [M_d(\bfj)] = \! N^{-N}\!\sum_{j_1}^N \ldots \sum_{j_N}^N \Vert \bfeta_d^{j_\ell}\Vert^2 = N^{-1}\sum_{j_\ell=1}^N  \Vert\bfeta_d^{j_\ell}\Vert^2$; for $\ell\not = \ell'$, it can bee seen that
$N^{-N}\!\sum_{\bfj\in\curJ} [M_d(\bfj)]$ $ =N^{-N+2} \sum_{\bfj\not = \{\bfj_\ell\cap\bfj_{\ell'}\}}\!
< \! N^{-1}\!\sum_{j_\ell=1}^N  \bfeta_d^{j_\ell} \, , N^{-1}\sum_{j_{\ell'}=1}^N  \bfeta_d^{j_{\ell'}} \!>  = 0$ due to equation \eqref{eq:7.3a};
grouping the two cases yields equation \eqref{eq:7.3b}.
(iii) Proof of equation \eqref{eq:7.3c}: this result can easily be deduced from equations \eqref{eq:7.2b}, \eqref{eq:7.3a}, and \eqref{eq:7.3b}.
\end{proof}
\begin{corollary}[\textbf{Value of} $d_N^2(N)$ \textbf{as a corollary of} Proposition~\ref{proposition:7.3}] \label{corollary:7.5}
Taking $m=N$ in equation \eqref{eq:7.2a} yields the value $d_N^2(N) = 1 + N/(N-1)$ given in Lemma~\ref{lemma:7.2}.
\end{corollary}
\begin{proof}
The proof is easy to obtain using Proposition~\ref{proposition:6.7}-(ii), Lemma~\ref{lemma:5.4}-(ii),   equations \eqref{eq:7.2a} and \eqref{eq:7.2b} with $m=N$, and equation \eqref{eq:3.2}.
\end{proof}
\begin{definition}[\textbf{Matrix} ${\relax{[\eta_d^m]}}$ \textbf{and function} $\varepsilon_d$]\label{definition:7.6}
Let $[\eta_d]\in \MM_{\nu,N}$ be the matrix defined by equation \eqref{eq:4.2} and let $[G_m]\in\MM_N $ be the matrix  defined in Definition~\ref{definition:5.3}. For all $m$ such that $1\leq m\leq N$, the matrix $[\eta_d^m]\in\MM_N$ is defined by
\begin{equation}\label{eq:7.4}
 [\eta_d^m]  =[\eta_d]\, [G_m] \, .
\end{equation}
The function $m\mapsto \varepsilon_d(m)$ with values in $\RR^+$ is defined by
\begin{equation}\label{eq:7.5}
  \Vert [\eta_d] - [\eta_d^m]\Vert = \varepsilon_d(m)\, \Vert \eta_d\Vert \, .
\end{equation}
\end{definition}
\begin{lemma}[\textbf{Properties of } $m\mapsto\varepsilon_d(m)$ \textbf{and expression of} ${\relax{\Vert \eta_d^m\Vert^2}}$]\label{lemma:7.7}
(i) $m\mapsto \varepsilon_d(m)$ is a decreasing function from $\{1,\ldots , N\}$ into  $[0\, , 1]$ and we have
$\varepsilon_d(1) = 1$ and $\varepsilon_d(N) = 0$. (ii) For all $m$ such that $1\leq m\leq N$, the square of the Frobenius norm of matrix $[\eta_d^m]$ that is defined by equation \eqref{eq:7.4}, is written as
\begin{equation}\label{eq:7.6}
  \Vert \eta_d^m\Vert^2 = (1-\varepsilon_d(m)^2)\, \Vert \eta_d\Vert^2\, .
\end{equation}
\end{lemma}
\begin{proof}
(i) From equation \eqref{eq:7.4} and Lemma~\ref{lemma:5.4}-(ii), it can be seen that $[\eta_d^N] = [\eta_d]$. Therefore, equation \eqref{eq:7.5} yields $\varepsilon_d(N)=0$. From \eqref{sec:5.1}, it can be seen that $\bfg^1= N^{-1/2}\Vert \bfpsi^1\Vert \,\bfun$ and using Definition~\ref{definition:5.3} yield $[G_1] = N^{-1}\, \bfun\otimes\bfun$. We have
$\Vert [\eta_d] - [\eta_d^1]\Vert^2=\Vert[\eta_d]([I_N] - [G_1])\Vert^2 = <\! [\eta_d]([I_N] - [G_1])\, , [\eta_d]([I_N] - [G_1])\!>_F
= $ $< \! [\eta_d]^T\, [\eta_d]\, , [I_N] - [G_1]\! >_F$ because $([I_N] - [G_1])^2 \! = [I_N] - 2[G_1] + [G_1]^2 = [I_N] - [G_1]$. Hence,
$\Vert  [\eta_d]  - [\eta_d^1] \Vert^2 \! = \Vert \eta_d \Vert^2 \! - N^{-1} \! < \! [\eta_d]^T  [\eta_d]\, , \bfun\otimes\bfun \! >_F$ and
$ < \![\eta_d]^T\, [\eta_d]\, , \bfun\otimes\bfun \! >_F = \Vert \sum_{j=1}^N \bfeta_d^j \Vert^2 = 0$
(due to the first equation \eqref{eq:2.3}), we deduce that $\Vert [\eta_d] - [\eta_d^1]\Vert^2 = \Vert\eta_d\Vert^2$, which proves that $\varepsilon_d(1) = 1$. Since $[G_m] = \sum_{\alpha=1}^m \bfvarphi^\alpha\otimes\bfvarphi^\alpha$ (see Lemma~\ref{lemma:5.4}-(iv)), we have $\Vert [\eta_d] - [\eta_d^m]\Vert^2
=< \! [\eta_d]^T\, [\eta_d]\, , [I_N] - [G_m]\! >_F
=\Vert\eta_d\Vert^2 -\sum_{\alpha=1}^m \Vert [\eta_d]\, \bfvarphi^\alpha\Vert^2$, which proves that $m\mapsto \varepsilon_d(m)$ is a decreasing function.
(ii) Developing the left-hand side of equation \eqref{eq:7.5} yields
$\Vert\eta_d\Vert^2 - 2 <\![\eta_d] \, , [\eta_d^m] \! >_F + \Vert\eta_d^m\Vert^2 = \varepsilon_d(m)^2 \Vert\eta_d\Vert^2$.
On the other hand, $<\![\eta_d] \, , [\eta_d^m] \! >_F  = <\![\eta_d] \, , [\eta_d]\, [G_m] \! >_F $ and since $[G_m] = [G_m]^2$ (see Lemma~\ref{lemma:5.4}-(iii)), it can be deduced that, $<\![\eta_d] \, , [\eta_d^m] \! >_F  = \Vert\eta_d^m\Vert^2$, and then equation \eqref{eq:7.6} holds.
\end{proof}
In Hypothesis~\ref{hypothesis:5.2}, based on the properties of $\{\lambda_\alpha(\varepsilon_\DM)\}_\alpha$, we have introduced the existence of optimal values $\varepsilon_\opt$ and $m_\opt$ of $\varepsilon_\DM$ and $m$, respectively. In the following hypothesis, we introduce the connection between $m\mapsto \varepsilon_d(m)$ and $\{\lambda_\alpha(\varepsilon_\DM)\}_\alpha$.
\begin{hypothesis}[\textbf{Relative to} $m\mapsto \varepsilon_d(m)$] \label{hypothesis:7.8}
Under Hypothesis~\ref{hypothesis:5.2}, it is assumed that $m_\opt$ is such that
$1 = \varepsilon_d(1) >  \ldots >\varepsilon_d(m_\opt-1) \gg \varepsilon_d(m_\opt) > \varepsilon_d(m_\opt+1)
> \ldots > \varepsilon_d(N) = 0$.
\end{hypothesis}
\begin{remark}[\textbf{Comments about hypotheses} \ref{hypothesis:5.2} \textbf{and} \ref{hypothesis:7.8} \textbf{related to} $m_\opt$]\label{remark:7.9}
Regarding Hypothesis~\ref{hypothesis:5.2} devoted to the existence of the optimal values $\varepsilon_\opt$ of $\varepsilon_\DM$ and  $m_\opt$ of $m$, \eqref{figure:1} (left) displays the graph of function $\alpha \mapsto \log(\lambda_\alpha(\varepsilon_\opt))$.
The graph of function $m\mapsto \varepsilon_d(m)$ corresponding to Hypothesis~\ref{hypothesis:7.8} is shown in \eqref{figure:1} (right). For $m> m_\opt$, we have $\varepsilon_d(m) \ll 1$ while $\varepsilon_d(1)=1$ and $\varepsilon_d(N)=0$.
\begin{figure}[htbp]
  \centering
  \includegraphics[width=6.0cm]{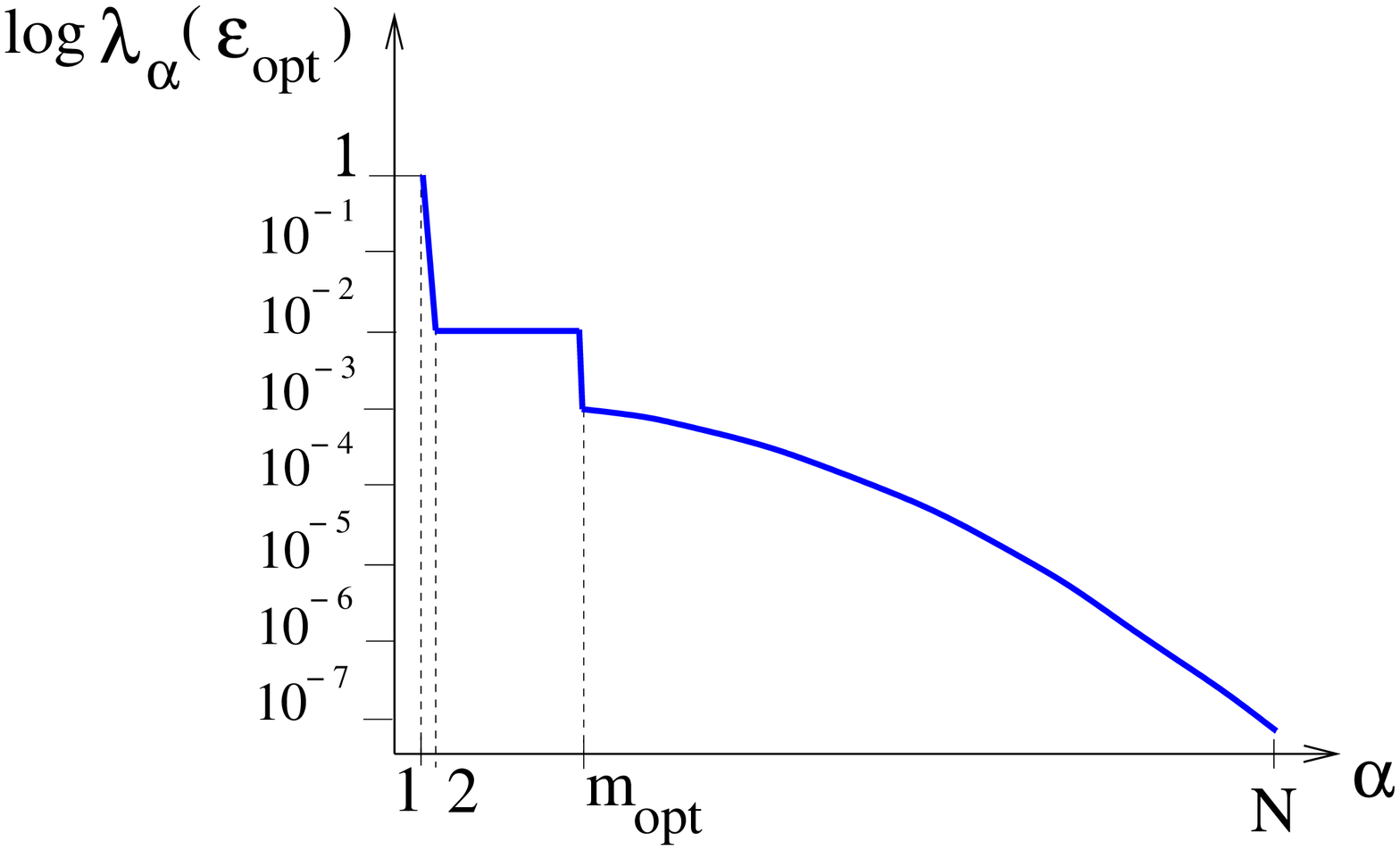} \includegraphics[width=5.2cm]{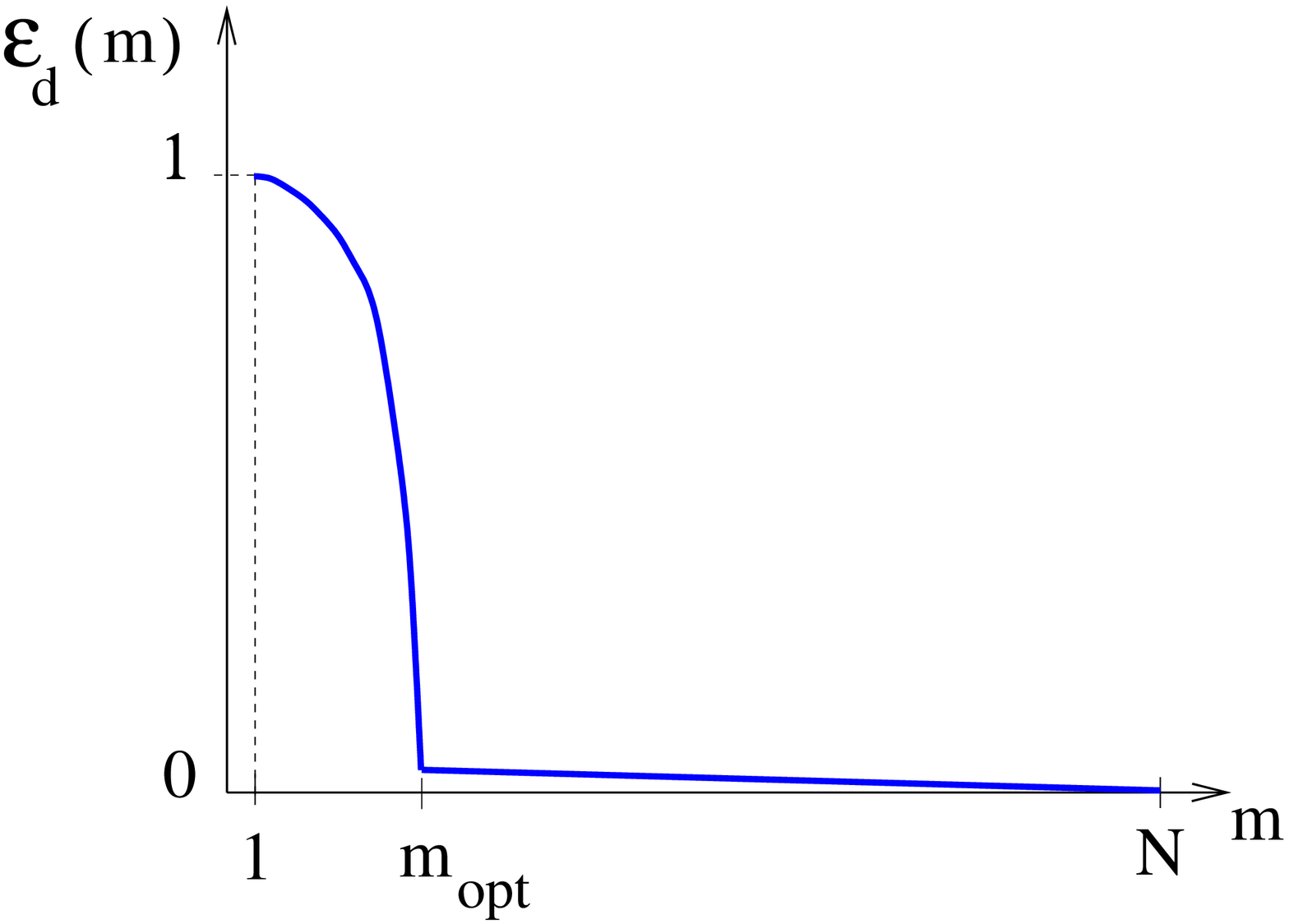}
  \caption{Left figure: for $\varepsilon_\DM = \varepsilon_\opt$, distribution of the eigenvalues $\lambda_\alpha(\varepsilon_\opt)$ in log scale as a function of rank $\alpha$. Right figure: graph of function $m\mapsto \varepsilon_d(m)$.}
  \label{figure:1}
\end{figure}
\end{remark}
\begin{proposition}[\textbf{Adapted expression of} $d_N^2(m)$]\label{proposition:7.10}
For all $m$ such that $1\leq m\leq N$, $d_N^2(m)$ given by equation \eqref{eq:7.2a} can be written as
\begin{subequations}\label{eq:7.7}
\begin{align}
  d_N^2(m) & = f_d(m) + h_d(m)\, , \label{eq:7.7a} \\
  f_d(m)   & = \frac{m\, \widehat s^2}{N-1} + \varepsilon_d(m)^2\, , \label{eq:7.7b} \\
  h_d(m)   & = \sum_{\bfj\in\curJ} p_\bfj(m)\, \frac{1}{\Vert\eta_d\Vert^2} \, \Vert [\eta_d^m] - \frac{\widehat s}{s} [\eta_d^m(\bfj)]\Vert^2\, , \label{eq:7.7c}
\end{align}
\end{subequations}
in which, $\forall \, \bfj\in\curJ$, $p_\bfj(m)$ is defined by equation \eqref{eq:6.8a} as a function of $\gamma_\bfj(m)$ (defined by equation \eqref{eq:6.8b}) that can be rewritten as
\begin{equation}\label{eq:7.8}
 \gamma_\bfj(m)  = \exp\{ -\frac{1}{2\, s^2} \Vert [\eta_d(\bfj)] - [\eta^m_d(\bfj)] \Vert^2\} \, ,
\end{equation}
where $[\eta^m_d(\bfj)] \in \MM_{\nu,N}$ is defined by
\begin{equation}\label{eq:7.9}
 [\eta^m_d(\bfj)]  = [\eta_d(\bfj)]\, [G_m]\, ,
\end{equation}
with $[\eta_d(\bfj)]\in\MM_{\nu,N}$ defined by equation \eqref{eq:4.3}.
\end{proposition}
\begin{proof}
(i) From equations \eqref{eq:7.2b} and \eqref{eq:6.9}, it can be seen that
$< \! [G_m] \, ,  [B_d(\bfj)] \! >_F $ $ =
(\widehat s / s)^2 <\! [G_m]\, ,  [\eta_d(\bfj)]^T [\eta_d(\bfj)] \!>_F - 2\, (\widehat s / s) <\! [G_m]\, ,  [\eta_d(\bfj)]^T [\eta_d] \!>_F $.
Since $[G_m] = [G_m]^2$ (see Lemma~\ref{lemma:5.4}-(iii)), we obtain
$< \! [G_m] \, ,  [B_d(\bfj)] \! >_F =
(\widehat s / s)^2 \Vert [ \eta_d(\bfj)]\,$ $[G_m] \Vert^2  - 2\, (\widehat s / s) <\! [\eta_d(\bfj)]\,[G_m]\, ,   [\eta_d] \,[G_m]\!>_F $,
which can be rewritten, using equations \eqref{eq:7.4} and \eqref{eq:7.9}, as
$< \! [G_m] \, ,  [B_d(\bfj)] \! >_F =
 \Vert  [\eta^m_d]  -  (\widehat s / s) \,[\eta^m_d(\bfj)]\Vert^2 - \Vert\eta^m_d\Vert^2$.
 By substitution into equation \eqref{eq:7.2a}, since $\sum_{\bfj\in\curJ} p_\bfj(m)=1$, and using equation \eqref{eq:7.6} yield
 equations \eqref{eq:7.7a} to \eqref{eq:7.7b}.
 (ii) We have $([I_N] - [G_m])^2 = [I_N] -2\, [G_m] + [G_m]^2 = [I_N] - [G_m]$.
 Consequently,  $< \! [I_N] - [G_m]\, , [M_d(\bfj)] \!>_F = $ $< \! ([I_N] - [G_m])^2\, , [\eta_d(\bfj)]^T [\eta_d(\bfj)] \!>_F
 =\Vert [\eta_d(\bfj)]\, ([I_N ]- [G_m])\Vert^2$. Using equation \eqref{eq:7.9} allows
 $< \! [I_N] - [G_m]\, , [M_d(\bfj)] \!>_F $ $= \Vert [\eta_d(\bfj)]  - [\eta^m_d(\bfj)]\Vert^2$
 to be written.
 By substitution  into equation \eqref{eq:6.8b} yields \eqref{eq:7.8}.
\end{proof}
\begin{lemma}[\textbf{Rewriting function} $h_d$]\label{lemma:7.11}
For all $m$ such that $1\leq m\leq N$ and for all $\bfj$ in $\curJ$, let $g_\bfj(m)$, $\underline g(m)$, $\underline \gamma(m)$, and $r(m)$ be defined by
\begin{equation}\label{eq:7.10}
 g_\bfj(m) = \!\frac{\Vert [\eta_d^m] \! - \!\frac{\widehat s}{s} [\eta_d^m(\bfj)]\Vert^2}{\Vert\eta_d\Vert^2} \, , \,\,
 \underline g(m) =\! \frac{1}{N^N}\sum_{\bfj\in\curJ} g_\bfj(m) \, , \,\,
 \underline \gamma(m)=\!\frac{1}{N^N}\sum_{\bfj\in\curJ}\gamma_\bfj(m) \, ,
\end{equation}
in which $\gamma_\bfj(m)$ is given by equation \eqref{eq:7.8}, and
\begin{equation}\label{eq:7.11}
  r(m) \, {\buildrel\hbox{\ppppcarac def}\over{=}} \,\sum_{\bfj\in\curJ} p_\bfj(m)\, \frac{g_\bfj(m)}{\underline g(m)} =
                  \frac{\frac{1}{N^N}\sum_{\bfj\in\curJ} \gamma_\bfj(m)\, g_\bfj(m) }{\underline \gamma(m)\,\underline g(m)} \, ,
\end{equation}
in which  $p_\bfj(m)$ is defined by equation \eqref{eq:6.8a}. Using these definitions,
we have
\begin{equation}\label{eq:7.12}
 \underline g(m) = 1 +\frac{\widehat s^2}{s^2}\frac{m}{N} -\varepsilon_d(m)^2 \quad , \quad  \underline g(m) > 0 \quad , \quad
 \underline g(N) = 1 +\frac{\widehat s^2}{s^2}\, ,
\end{equation}
in which $\varepsilon_d(m)$ is defined by equation \eqref{eq:7.5}, and
\begin{equation}\label{eq:7.13}
 \underline \gamma(m)> 0 \quad , \quad   \underline \gamma(N) =1 \quad , \quad r(m) >0 \quad , \quad  r(N) = 1 \, .
\end{equation}
Function $m\mapsto h_d(m)$ defined by equation \eqref{eq:7.7c} can be rewritten as
\begin{equation}\label{eq:7.14}
h_d(m) = r(m)\, \underline g(m)\, .
\end{equation}
\end{lemma}
\begin{proof}
Using equation \eqref{eq:7.10} yields
$\underline g(m) = \Vert\eta_d\Vert^{-2} \{ N^{-N}\sum_{\bfj\in\curJ}\Vert\eta_d^m\Vert^2 +
    (\widehat s^2 / s^2)$ $ N^{-N}  \sum_{\bfj\in\curJ}\Vert\eta_d^m(\bfj)\Vert^2
    - 2 (\widehat s /s) < \! [\eta_d^m] \, , N^{-N}\sum_{\bfj\in\curJ} [\eta_d^m(\bfj)]\! >_F   \}$.
It can be seen that $\Vert \eta_d^m(\bfj)\Vert^2 = <\! [\eta_d(\bfj)]\, [G_m]\, , [\eta_d(\bfj)]\, [G_m]\! >_F$. Since $[G_m]^2=[G_m]$ (Lemma~\ref{lemma:5.4}-(iii)) and
$[\eta_d(\bfj)]^T [\eta_d(\bfj)] = [M_d(\bfj)]$ (equation \eqref{eq:6.9}), we have
$N^{-N} \!\sum_{\bfj\in\curJ}\Vert\eta_d^m(\bfj)\Vert^2 = $ $<\! N^{-N} \sum_{\bfj\in\curJ}[M_d(\bfj)]\, , [G_m]\! >_F$
that can be rewritten, using \eqref{eq:7.3b} and Lemma~\ref{lemma:5.4}-(i), as
$N^{-N} \!\! \sum_{\bfj\in\curJ}\! \Vert\eta_d^m(\bfj)\Vert^2 $ $= N^{-1} \Vert\eta_d\Vert^2 \,\tr\{ [G_m]\} = (m/N)\Vert\eta_d\Vert^2$.
 Substituting $\Vert\eta_d^m\Vert^2$ given by equation \eqref{eq:7.6} in the above expression of $\underline g(m)$ and using equation \eqref{eq:7.3a}
yield equation \eqref{eq:7.12}.
 Due to equation \eqref{eq:7.10}, $\underline g(m) > 0$ and since $\varepsilon_d(N)=0$ (see Lemma~\ref{lemma:7.7}-(i)), the first equation \eqref{eq:7.12} with $m=N$ yields the third equation \eqref{eq:7.12}.
Since $[G_N] = [I_N]$ (see Lemma~\ref{lemma:5.4}-(ii)) and using the definition of $\gamma_\bfj(m)$ given by equation \eqref{eq:6.8b} yield $\gamma_\bfj(N)=1$.
 The other results of the lemma are easy to prove.
\end{proof}
\begin{lemma}[\textbf{Property of function} $f_d$] \label{lemma:7.12}
For  $N$ and $\nu$ fixed, let $\widehat s$ be defined by equation \eqref{eq:3.2}.
Function $m\mapsto f_d(m)$ from $\{1,\ldots , N\}$ into $\RR^+$, defined by equation \eqref{eq:7.7b}, is such that
$f_d(1) = 1 +\widehat s^2/(N-1)$ and $f_d(N) = N\,\widehat s^2 / (N-1)$.  Let $m_\opt$ be the value of $m$ defined in Hypothesis~\ref{hypothesis:5.2}.
If
\begin{equation}\label{eq:7.15}
  \varepsilon_d(m_\opt)^2 < \frac{\widehat s^2}{N-1} < \varepsilon_d(m_\opt -1)^2 \, ,
\end{equation}
then function $m\mapsto f_d(m)$ has a unique local minimum that is a global minimum, which is reached for $m=m_\opt$,
\begin{equation}\label{eq:7.16}
  m_\opt = \arg \min_{1\leq m\leq N} f_d(m)\, .
\end{equation}
\end{lemma}
\begin{proof}
The value of $f_d(1)$ and $f_d(N)$ are directly deduced from equation \eqref{eq:7.7b} and also from the values $\varepsilon_d(1)=1$ and $\varepsilon_d(N)=0$ (see Lemma~\ref{lemma:7.7}).
(i) Let $m$ be such that $m_\opt \leq m\leq N-1$, and let
$\Delta_m^+ = f_d(m+1) - f_d(m) = \widehat s^2/(N-1) +\varepsilon_d(m+1)^2 - \varepsilon_d(m)^2$.
Since $\varepsilon_d$ is a decreasing function (see Lemma~\ref{lemma:7.7}), $\varepsilon_d(m)^2\leq \varepsilon_d(m_\opt)^2$, and consequently,
$\Delta_m^+ \geq \widehat s^2/(N-1) +\varepsilon_d(m+1)^2 - \varepsilon_d(m_\opt)^2$.
Since $\varepsilon_d(m+1)^2 > 0$ and since $\widehat s^2/(N-1)-\varepsilon_d(m_\opt)^2 > 0$ (due to equation \eqref{eq:7.15}), we have
$\Delta_m^+ > 0$ and therefore, $f_d$ is an increasing function on $\{m_\opt,\ldots , N\}$.
(ii) Let $m$ be such that $1 \leq m\leq m_\opt$, and let
$\Delta_m^- = f_d(m) - f_d(m-1) = \widehat s^2/(N-1) +\varepsilon_d(m)^2 - \varepsilon_d(m-1)^2$.
From Hypothesis~\ref{hypothesis:7.8}, it can be deduced that $\varepsilon_d(m)^2 - \varepsilon_d(m-1)^2 < 0$.
For all $2 \leq m\leq m_\opt$, we have
$\varepsilon_d(m-1)^2 \geq \varepsilon_d(m_\opt -1)^2 \gg \varepsilon_d(m_\opt)^2$ and equation \eqref{eq:7.15} shows that  $\Delta_m^- <0$.
(iii) Since $\Delta_m^- <0$ for $1\leq m \leq m_\opt$ and $\Delta_m^+ > 0$ for $m_\opt \leq m \leq N$ yield equation \eqref{eq:7.16}.
\end{proof}
\begin{theorem}[\textbf{Existence of a minimum of} $d_N^2(m)$ \textbf{for} $m < N$]\label{theorem:7.13}
Let $\curM_\opt =\{m_\opt,$ $m_\opt+1,\ldots , N\}$ in which $m_\opt$ is defined in Hypothesis~\ref{hypothesis:5.2}.
\begin{subequations}\label{eq:7.17}
\begin{align}
  \hbox{If}   & \quad  \forall \, m \in \curM_\opt\quad , \quad r(m) \leq 1 \, , \label{eq:7.17a} \\
  \hbox{then} & \quad \min_{m\in\curM_\popt} d_N^2(m) \leq \min_{m\in\curM_\popt} d_N^{2,\psup}(m) < d_N^2(N) \, ,\label{eq:7.17b}
\end{align}
\end{subequations}
in which $d_N^{2,\psup}(m)$ is written as
\begin{equation}\label{eq:7.18}
  d_N^{2,\psup}(m) = 1 +\frac{m}{N-1} \, .
\end{equation}
Equations \eqref{eq:7.17a} and \eqref{eq:7.17b} shows that
\begin{equation}\label{eq:7.19}
 \min_{m\in\curM_\popt} d_N^2(m)  \leq 1 +\frac{m_\opt}{N-1} < d_N^2(N) \, ,
\end{equation}
which means that, if hypothesis defined by equation \eqref{eq:7.17a} holds, then the PLoM method is a better method than the usual one corresponding to $d_N^2(N)$.
\end{theorem}
\begin{proof}
Equations \eqref{eq:7.7a} and \eqref{eq:7.14} yield $d_N^2(m) = f_d(m)+r(m)\, \underline g(m)$.
If $r(m) \leq 1$ for all $m$ in $\curM_\opt$, then
\begin{equation}\label{eq:7.20}
 d_N^2(m)  \leq d_N^{2,\psup}(m) \quad , \quad \forall\, m \in \curM_\opt \, ,
\end{equation}
in which $d_N^{2,\psup}(m) = f_d(m) + \underline g(m)$. From equations \eqref{eq:7.7b} and \eqref{eq:7.12}, it can be seen that
$d_N^{2,\psup}(m) = 1 +m\,\widehat s^2/(N-1) + (\widehat s^2 /s^2)(m/N)$ that can be rewritten, using equation \eqref{eq:3.2}, as equation \eqref{eq:7.18}.
For $m=N$, equation \eqref{eq:7.18} and Lemma~\ref{lemma:7.2} yield
$d_N^{2,\psup}(N) = d_N^2(N) = 1 + N/(N-1)$. From equations \eqref{eq:7.18} and \eqref{eq:7.20}, it can then be deduced that
$\min_{m\in\curM_\popt} d_N^2(m) \leq \min_{m\in\curM_\popt} d_N^{2,\psup}(m) =1 + m_\opt/(N-1)$, and since $d_N^2(N) = 1 + N/(N-1)$, we obtain equation \eqref{eq:7.19}.
\end{proof}
\begin{remark}[\textbf{Concerning the hypothesis} $r(m)\leq 1, \forall m\in \curM_\opt$]\label{remark:7.14}
For $m\in\{1, \ldots,$ $N\}$, $r(m)$ defined by equation \eqref{eq:7.11} does not seem to be calculable either explicitly or numerically (since there are $N^N$ elements in set $\curJ$). This is the reason why we have introduced the hypothesis defined by equation \eqref{eq:7.17a} in order to formulate the theorem. Obviously, this hypothesis has numerically been verified by a direct Monte Carlo simulation of $d_N^2(m)$ given in Definition~\ref{definition:7.1} using \eqref{sec:6.3} and Remark~\ref{remark:6.8}. In \eqref{sec:8}, we give additional developments and comments about this hypothesis.
\end{remark}
\section{Justification of the  hypothesis introduced in Theorem~\ref{theorem:7.13}}
\label{sec:8}
As explained in Remark~\ref{remark:7.14}, $r(m)$ defined by equation \eqref{eq:7.11}  cannot explicitly be calculated for $1\leq m\leq N-1$ (for $m=N$, we have $r(N) = 1$).
In this section, we give a preliminary remark showing the difficulty. We then propose  an estimation of $r(m)$ using the maximum entropy principle from Information Theory, and finally, we propose a rough approximation of $r(m)$. In \eqref{sec:9}, devoted to a numerical illustration, we will compare the two last estimations of $d_N^2(m)$ with the "true" function $d_N^2(m)$ estimated as explained in Remark~\ref{remark:7.14}.
\begin{remark}[\textbf{Preliminary remark}]\label{remark:8.1}
For all $\bfj$ in $\curJ$ and for all $m$ in $\{1,\ldots , N\}$, let $a_\bfj(m) \geq 0$ be defined in Proposition~\ref{proposition:6.7}-(ii).
Therefore, $\gamma_\bfj(m)$, which is defined by equation \eqref{eq:6.8b}, can be rewritten as
$\gamma_\bfj(m) = \exp\{-\frac{1}{2 s^2} \, a_\bfj(m)\}$ and consequently, $p_\bfj (m)$ defined by equation \eqref{eq:6.8a}, can also be rewritten as
\begin{equation}\label{eq:8.1}
  p_\bfj(m) = \frac{ \exp\{-\frac{1}{2s^2}  \, a_\bfj(m)\} }
                   { \sum_{\bfj'\in\curJ} \exp\{-\frac{1}{2s^2}  \, a_{\bfj'}(m)\}} \quad , \quad \sum_{\bfj\in\curJ} p_\bfj(m) = 1 \, .
\end{equation}
The discrete random variable $A(m)$ with values in $\{ a_\bfj(m),\bfj\in\curJ\}$, whose probability distribution $\{p_\bfj(m),\bfj\in\curJ\}$ is defined by equation \eqref{eq:8.1}, is a Maxwell-Boltzmann distribution. Its mean value is $\underline a(m) = E\{A(m)\} = \sum_{\bfj\in\curJ} a_\bfj(m)\, p_\bfj(m)$, and its entropy is written as $S( \{ p_\bfj(m) \}_\bfj ) = -\sum_{\bfj\in\curJ} p_\bfj(m)\, \log p_\bfj(m) = \frac{1}{2 s^2}\, \underline a(m) + \log (\sum_{\bfj\in\curJ} \exp\{-\frac{1}{2s^2}  \, a_{\bfj}(m)\})$.
From equation \eqref{eq:7.11}, we then have to calculate, for all $m\in\{1,\ldots, N-1\}$,
$r(m) = \sum_{\bfj\in\curJ} (g_\bfj(m) / \underline g(m))\, p_\bfj(m)$, in which $g_\bfj(m)$ is defined by equation \eqref{eq:7.10}.
As we have explained, such a calculation cannot be performed neither explicitly nor numerically (there are $N^N$ elements in $\curJ$).
\end{remark}

In the following, we construct an estimation of $r(m)$ using the maximum entropy principle.

\begin{definition}[\textbf{Discrete random matrix} ${\relax{ [\bfA]}}$]\label{definition:8.2}
Let $[\bfA]$ be the discrete random variable with values  in $\{[\eta_d(\bfj)] ,\bfj\in\curJ\}$ with $[\eta_d(\bfj)]\in \MM_{\nu,N}$ defined by equation \eqref{eq:4.3}, and for which the probability distribution is $\{\widehat p_\bfj ,\bfj\in\curJ\}$ with $\widehat p_\bfj = 1/N^N$,
\begin{equation}\label{eq:8.2}
  P_{[\bfA]}(d[a]) = \sum_{\bfj\in\curJ} \widehat p_\bfj\,  \delta_{0_{\MM_{\nu,N}}}([a] - [\eta_d(\bfj)]) \, .
\end{equation}
\end{definition}
\begin{lemma}[\textbf{Second-order moments of random matrix} ${\relax{ [\bfA] }}$]\label{lemma:8.3}
\begin{equation}\label{eq:8.3}
 E\{[\bfA]\}  =[0_{\nu,N}] \quad , \quad E\{[\bfA]^T\,[\bfA]\} = \frac{1}{N} \, \Vert\eta_d\Vert^2\, [I_N]\, .
\end{equation}
\end{lemma}
\begin{proof}
We have
$E\{[\bfA]\} = \int_{\MM_{\nu,N}} [a] \sum_{\bfj\in\curJ}\widehat p_\bfj \,\delta_{0_{\MM_{\nu,N}}}([a] - [\eta_d(\bfj)])$ and
$E\{[\bfA]^T[\bfA]\}$ $= \int_{\MM_{\nu,N}} [a]^T\, [a]\sum_{\bfj\in\curJ}\widehat p_\bfj \,  \delta_{0_{\MM_{\nu,N}}}([a] - [\eta_d(\bfj)])$ yielding
$E\{[\bfA]\} = N^{-N}\sum_{\bfj\in\curJ} [\eta_d(\bfj)]$ and $E\{[\bfA]^T\,[\bfA]\} = N^{-N}\sum_{\bfj\in\curJ} [\eta_d(\bfj)]^T \, [\eta_d(\bfj)]$.
Using equations \eqref{eq:7.3a} and \eqref{eq:6.9} with equation \eqref{eq:7.3b} yield equation \eqref{eq:8.3}.
\end{proof}
\begin{lemma}[\textbf{Expression of} $r(m)$ \textbf{as a function of random matrix} ${\relax{[\bfA]}}$]\label{lemma:8.4}
Let $m$ be fixed in $\{1,\ldots , N\}$. Function $h_d(m) = r(m)\, \underline g(m)$  defined by equation \eqref{eq:7.14} in which $r(m)$ is defined by equation \eqref{eq:7.11} can be rewritten as
\begin{subequations}\label{eq:8.4}
  \begin{align}
    r(m)\, \underline g(m) & = 1 \! - \!\varepsilon_d(m)^2 \! + \frac{1}{\underline \gamma(m)\, \Vert\eta_d\Vert^2}
                              <\! \frac{\widehat s^2}{s^2}[T_2(m)] - 2 \frac{\widehat s}{s}[T_1(m)]\, [\eta_d]\, , [G_m]\! >_F \, , \label{eq:8.4a}\\
    \underline \gamma(m)   & = E\{ \exp(-\frac{1}{2 s^2} <\![I_N] - [G_m] \, , [\bfA]^T \, [\bfA]\! >_F )\} \, , \label{eq:8.4b}\\
    [T_1(m)]               & = E\{ [\bfA]^T \exp(-\frac{1}{2 s^2} <\![I_N] - [G_m] \, , [\bfA]^T \, [\bfA]\! > )\}\, , \label{eq:8.4c}\\
    [T_2(m)]               & = E\{ [\bfA]^T [\bfA] \exp(-\frac{1}{2 s^2} <\![I_N] - [G_m] \, , [\bfA]^T \, [\bfA]\! >_F )\}\, . \label{eq:8.4d}
  \end{align}
\end{subequations}
\end{lemma}
\begin{proof}
For all $\bfj$ in $\curJ$, $g_\bfj(m)$ (defined by equation \eqref{eq:7.10}) can be rewritten, using the proof of Proposition~\ref{proposition:7.10} and equation \eqref{eq:7.2b}, as
follows, $g_\bfj(m) = \Vert\eta_d\Vert^{-2} \{ \Vert\eta_d^m\Vert^2  +  $ $<\! (\widehat s^2 / s^2) [M_d(\bfj)] - 2(\widehat s /s) [\eta_d(\bfj)]^T [\eta_d] \, , [G_m]\!>_F\}$, or using equation \eqref{eq:7.6},
$g_\bfj(m) = 1 - \varepsilon_d(m)^2 + \Vert\eta_d\Vert^{-2}\{ <\! (\widehat s^2 / s^2) [\eta_d(\bfj)]^T [\eta_d(\bfj)]- 2(\widehat s /s) [\eta_d(\bfj)]^T [\eta_d] \, , [G_m]\!>_F\}$. Substituting this expression of $g_\bfj(m)$ into equation \eqref{eq:7.11} and using equation \eqref{eq:7.10} yield equation \eqref{eq:8.4a} in which
$\underline \gamma(m)  = \sum_{\bfj\in\curJ} \widehat p_\bfj \, \exp(-\frac{1}{2 s^2} <\![I_N] - [G_m] \, , [\eta_d(\bfj)]^T$ $[\eta_d(\bfj)]\! >_F )\}$, where
$[T_1(m)]  = \sum_{\bfj\in\curJ} \widehat p_\bfj \, [\eta_d(\bfj)]^T \exp(-\frac{1}{2 s^2} <\![I_N] - [G_m] \, , [\eta_d(\bfj)]^T$ $[\eta_d(\bfj)]\! >_F )\}$, and where
$[T_2(m)]  = \sum_{\bfj\in\curJ} \widehat p_\bfj \, [\eta_d(\bfj)]^T [\eta_d(\bfj)] \exp(-\frac{1}{2 s^2} <\![I_N] - [G_m] \, , [\eta_d(\bfj)]^T [\eta_d(\bfj)]\! >_F )\}$.
Using equation \eqref{eq:8.2} allows  $\underline \gamma(m)$, $[T_1(m)]$, and $[T_2(m)]$ to rewritten as equations \eqref{eq:8.4b}, \eqref{eq:8.4c}, and \eqref{eq:8.4d}.
\end{proof}
Below, an approximation $[\bfA^c]$ of random matrix $[\bfA]$ is constructed using the maximum entropy principle \cite{Shannon1948,Jaynes1957a,Jaynes1957b,Cover2006,Soize2017b} under the available information defined by equation \eqref{eq:8.3}.
\begin{definition}[\textbf{Random matrix} ${\relax{[\bfA^c]}}$]\label{definition:8.5}
Let $[\bfA^c]$ be the random matrix with values in $\MM_{\nu,N}$ whose probability measure $P_{[\bfA^c]}(d[a])$ is defined by a pdf $[a]\mapsto p_{[\bfA^c]}([a])$ on $\MM_{\nu,N}$ with respect to $d[a]$. This pdf is constructed as the unique solution of the following maximum entropy (MaxEnt) problem,
\begin{equation}\label{eq:8.5}
 p_{[\bfA^c]} = \max_{p\in\curC_\pad} S(p)\, ,
\end{equation}
in which the entropy is written as $S(p)= -\int_{\MM_{\nu,N}} p([a])\, \log(p([a])) \, d[a]$ and where the admissible set is defined by
$\curC_\ad =\{ [a]\mapsto p([a]):\MM_{\nu,N}\rightarrow \RR^+$, $\int_{\MM_{\nu,N}} p([a])\, d[a]= 1$,
$ \int_{\MM_{\nu,N}} [a]\, p([a])\, d[a] = [0_{\nu,N}]$, $\int_{\MM_{\nu,N}} [a]^T\, [a]\, p([a])\, d[a] = (1/N) \Vert\eta_d\Vert^2\, [I_N] \}$.
\end{definition}
\begin{proposition}[\textbf{Explicit expression of pdf} ${\relax{p_{[\bfA^c]}}}$ ]\label{proposition:8.6}
The optimization problem defined by equation \eqref{eq:8.5} has a unique solution written, for all $[a]$ in $\MM_{\nu,N}$, as
\begin{equation}\label{eq:8.6}
  p_{[\bfA^c]}([a]) = \frac{1}{(2\pi)^{\nu N/2}} \frac{1}{\sigma^{\nu N}} \exp\{-\frac{1}{2\sigma^2}\Vert a\Vert^2\}
  \,\, , \,\, \sigma^2 = \frac{1}{\nu N} \Vert\eta_d\Vert^2 = 1- \frac{1}{N}\, .
\end{equation}
\end{proposition}
\begin{proof}
The proof is  left to the reader (see for instance, \cite{Soize2017b}).
\end{proof}
\begin{remark}[\textbf{Independence of the entries of random matrix} ${\relax{[\bfA^c]}}$]\label{remark:8.7}
Equation \eqref{eq:8.6} shows that the real-valued random variables $\{[\bfA^c]_{k\ell}\, ; k=1,\ldots , \nu\, ; \ell=1,\ldots , N\}$
are independent and, each one $\bfA^c_{k\ell} =[\bfA^c]_{k\ell}$, is a second-order, centered, Gaussian random variable for which its variance is $\sigma^2$,
\begin{equation}\label{eq:8.7}
 p_{\bfA^c_{k\ell}} (a_{k\ell} )= \frac{1}{\sqrt{2\pi}\, \sigma} \exp\{-\frac{1}{2\sigma^2} a_{k\ell}^2\}\quad , \quad
 E\{\bfA^c_{k\ell}\} = 0 \quad , \quad  E\{(\bfA^c_{k\ell})^2\} = \sigma^2\, .
\end{equation}
A simple calculation shows that $[\bfA^c]$ effectively satisfies the constraints defined by the available information, that is,
$E\{[\bfA^c]\} \! =[0_{\nu,N}]$ and $E\{[\bfA^c]^T[\bfA^c]\} = (1/N) \, \Vert\eta_d\Vert^2\, [I_N]$.
\end{remark}
\begin{remark}[\textbf{Comparison of the entropy of measures} ${\relax{P_{[\bfA]}}}$ \textbf{and} ${\relax{P_{[\bfA^c]}}}$]\label{remark:8.8}
Let us compare the entropy of $P_{[\bfA]}(d[a])$  defined by equation \eqref{eq:8.2} with the entropy of $P_{[\bfA^c]}(d[a])$ $= p_{[\bfA^c]}([a])\, d[a]$ whose pdf is defined by equation \eqref{eq:8.6}. Since $\widehat p_\bfj = 1/N^N$, we have
$(P_{[\bfA]}) = -\sum_{\bfj\in\curJ} \widehat p_\bfj \, \log \widehat p_\bfj = N\log N$. On the other hand, we have $ S( P_{[\bfA^c]} ) = -\int_{\MM_{\nu,N}} p_{[\bfA^c]} ([a])\, \log p_{[\bfA^c]}([a])$ $d[a] = (\nu N/2) ( \log (2 \pi e) + \log(1-1/N) )$.
Consequently, $ S( P_{[\bfA^c]} ) / S(P_{[\bfA]}) = \nu (2\log N)^{-1} (\log(2\pi e) + \log(1-1/N))$.
Clearly, the approximation will be optimal if $S( P_{[\bfA^c]} ) \sim  S(P_{[\bfA]})$, which, for $\nu\geq 2$, is reached if $N\sim (2\pi e)^{\nu/2}$.  In general, $N < (2\pi e)^{\nu/2}$ and consequently, the level of uncertainties associated with probability measure $P_{[\bfA^c]}(d[a])$ is larger than the one for the probability measure $P_{[\bfA]}$. For instance, in \eqref{sec:9} devoted to the numerical illustration, we have $\nu=9$ and $N=200$, which  yields
$ S( P_{[\bfA^c]} ) / S(P_{[\bfA]}) = 2.4$.
\end{remark}
\begin{definition}[\textbf{Approximation} $r^c(m)$ \textbf{of} $r(m)$]\label{definition:8.9}
For all $m$ in $\{1,\ldots , m\}$, the approximations $r^c(m)$, $\underline \gamma^c(m)$, $[T_1^c(m)]$, and $[T_2^c(m)]$ of $r(m)$, $\underline \gamma(m)$, $[T_1(m)]$, and $[T_2(m)]$  is obtained by replacing  $[\bfA]$ by $[\bfA^c]$ in equations \eqref{eq:8.4a} to \eqref{eq:8.4d},
\begin{subequations}\label{eq:8.8}
  \begin{align}
    r^c(m) \underline g(m) \! & = \!1 \! - \!\varepsilon_d(m)^2 \! + \!\frac{1}{\underline \gamma^c(m) \Vert\eta_d\Vert^2}
                             \! <\! \frac{\widehat s^2}{s^2}[T_2^c(m)]\! - \! 2 \frac{\widehat s}{s}[T_1^c(m)]\, [\eta_d]\, , [G_m]\! >_F  , \label{eq:8.8a}\\
    \underline \gamma^c(m)   & = E\{ \exp(-\frac{1}{2 s^2} <\![I_N] - [G_m] \, , [\bfA^c]^T \, [\bfA^c]\! >_F )\} \, , \label{eq:8.8b}\\
    [T_1^c(m)]               & = E\{ [\bfA^c]^T \exp(-\frac{1}{2 s^2} <\![I_N] - [G_m] \, , [\bfA^c]^T \, [\bfA^c]\! >_F )\}\, , \label{eq:8.8c}\\
    [T_2^c(m)]               & = E\{ [\bfA^c]^T [\bfA^c] \exp(-\frac{1}{2 s^2} <\![I_N] - [G_m] \, , [\bfA^c]^T \, [\bfA^c]\! >_F )\}\, . \label{eq:8.8d}
  \end{align}
\end{subequations}
\end{definition}
\begin{lemma}[\textbf{Explicit calculation of} ${\relax{\underline \gamma^c(m), [T_1^c(m)]}}$, \textbf{and} ${\relax{[T_2^c(m)]}}$]\label{lemma:8.10}
For all integer $m$ in $\{1,\ldots , N\}$, we have
\begin{equation}\label{eq:8.9}
 \underline \gamma^c(m)  = (1+\frac{\sigma^2}{s^2})^{-\nu(N-m)/2}\,\, , \,\, [T_1^c(m)]  =[0_{N,\nu}] \,\, , \,\, [T_2^c(m)]
                                           = \underline \gamma^c(m) \, \nu \,[b_m]^{-1} \,
\end{equation}
in which $\sigma^2$ is defined by the second equation \eqref{eq:8.6} and where the matrix $[b_m]$ is defined by $[b_m] = s^{-2}([I_N]-[G_m]) +\sigma^{-2} [I_N]$ and belongs to $\MM_N^+$ (and thus is invertible).
\end{lemma}
\begin{proof}
Let $f$ be a mapping on  $\MM_{\nu,N}$ such that the following quantity be defined,
$\curL_f(m)   = E\{ f([\bfA^c]) \, \exp(-\frac{1}{2 s^2} <\![I_N] - [G_m] \, , [\bfA^c]^T \, [\bfA^c]\! >_F )\}$.
Using equation \eqref{eq:8.6},  it can be seen that
$\curL_f(m)   =  (2\pi\sigma^2)^{-\nu N/2} \int_{\MM_{\nu,N}} f([a])\, \exp(-\frac{1}{2}\! <\![a]^T [a] \, , [b_m]\! >_F ) \, d[a]$.
From Lemma~\ref{lemma:5.4}-(iv), $[G_m] = \sum_{\alpha=1}^N \mu_\alpha \bfvarphi^\alpha \otimes \bfvarphi^\alpha$ with $<\! \bfvarphi^\alpha\, , \bfvarphi^\beta\! > =\delta_{\alpha\beta}$, $\mu_1=\ldots = \mu_m = 1$, and $\mu_{m+1}=\ldots =\mu_N = 0$.
Since $[I_N] =\sum_{\alpha=1}^N \bfvarphi^\alpha \otimes \bfvarphi^\alpha$, matrix $[b_m]$ can be rewritten as $[b_m] = \sum_{\alpha=1}^N \zeta_\alpha(m) \,\bfvarphi^\alpha \otimes \bfvarphi^\alpha$
in which  $\zeta_\alpha(m) = 1/\sigma^2$ if $\alpha\leq m$, and  $\zeta_\alpha(m) = 1/s^2 + 1/\sigma^2$ if $\alpha > m$.
Since $\zeta_\alpha(m) > 0, \forall\alpha$, it can be seen that $[b_m]\in\MM_N^+$. Hence,
$[b_m]^{-1} = \sum_{\alpha=1}^N (1/\zeta_\alpha(m)) \,\bfvarphi^\alpha \otimes \bfvarphi^\alpha$ and consequently,
$\det\{[b_m]^{-1}\}  = (\zeta_1(m)\!\times \ldots \times \!\zeta_N(m) ) ^{-1} = \sigma^{2N}\, (1 + \frac{\sigma^2}{s^2})^{-(N-m)}$.
Hence, $\curL_f(m)$ can be rewritten as
$\curL_f(m)   = (1 +\sigma^2/s^2)^{-\nu(N-m)/2 } \int_{\RR^N}\ldots \int_{\RR^N} f([\widehat a]^T)\, p(\widehat\bfa^1)\!\times\!\ldots\!\times\! p(\widehat\bfa^\nu)\,
d\widehat\bfa^1\ldots d\widehat\bfa^\nu$, where $\forall k\in\{1,\ldots,\nu\}$, $\widehat\bfa^k \in \RR^N$ is such that $\forall\ell\in\{1,\ldots , N\}$,
$\widehat a^k_\ell = \{[\widehat a]^T\}_{\ell k}  =  [a]_{k\ell}$, and where
$p(\widehat\bfa^k) = ((2\pi)^{N/2} \sqrt{\det\{[b_m]^{-1}\}})^{-1}\, \exp(-\frac{1}{2} \!<\! [b_m]\,\widehat\bfa^k \, , \widehat\bfa^k\! > )\!$
is the pdf of a Gaussian centered second-order $\RR^N$-valued random variable $\widehat\bfA^k$ whose covariance matrix is $[b_m]^{-1}$.
(i) Taking $f([a]) = 1$, equation \eqref{eq:8.8b} is written as
$\underline \gamma^c(m) = (1 +\sigma^2/s^2)^{-\nu(N-m)/2 }$  $\Pi_{k=1}^\nu\{\int_{\RR^N} p(\widehat\bfa^k)\, d \widehat\bfa^k\}$ that gives the first equation \eqref{eq:8.9}.
(ii) Taking $f([a]) = [a]^T$, equation \eqref{eq:8.8c} is written as
$[T_1^c(m)]   = (1 +\sigma^2/s^2)^{-\nu(N-m)/2 }\int_{\RR^N}\ldots \int_{\RR^N} [\widehat a] \, p(\widehat\bfa^1)\!\times\!\ldots\!\times\! p(\widehat\bfa^\nu)\,
d\widehat\bfa^1\ldots d\widehat\bfa^\nu$, which is equal to $[0_{N,\nu}]$ because $\widehat\bfA^k$ is centered, and therefore, the second equation \eqref{eq:8.9} is proven.
(iii) Finally, taking $f([a]) = [a]^T [a]$, equation \eqref{eq:8.8d} is written as
$[T_2^c(m)]   = (1 +\sigma^2/s^2)^{-\nu(N-m)/2 }\int_{\RR^N}\ldots \int_{\RR^N} [\widehat a] [\widehat a]^T\, p(\widehat\bfa^1)\!\times\!\ldots\!\times\! p(\widehat\bfa^\nu)\, d\widehat\bfa^1\ldots d\widehat\bfa^\nu$ whose entries are
$[T_2^c(m)]_{\ell\ell'}   = (1 +\sigma^2/s^2)^{-\nu(N-m)/2 }\!\sum_{k=1}^\nu\!\int_{\RR^N}\ldots \int_{\RR^N}$ $\widehat a^k_\ell \,  \widehat a^k_{\ell'} \, p(\widehat\bfa^1)\!\times\!\ldots\!\times\! p(\widehat\bfa^\nu)\, d\widehat\bfa^1\ldots d\widehat\bfa^\nu$, which shows that
$[T_2^c(m)]   = (1 +\sigma^2/s^2)^{-\nu(N-m)/2 }$ $\sum_{k=1}^\nu E\{\widehat\bfA^k\otimes \widehat\bfA^k\}$. The third equation \eqref{eq:8.9} is then directly deduced.
\end{proof}
\begin{proposition}[\textbf{Expression of} $r^c(m)$]\label{proposition:8.11}
For all $m$ in $\{1,\ldots , N\}$, we have
\begin{equation}\label{eq:8.10}
  r^c(m) = 1 \, .
\end{equation}
\end{proposition}
\begin{proof}
Substituting the second and the third equation \eqref{eq:8.9} into equation \eqref{eq:8.8a} yields
$ r^c(m)\, \underline g(m)  = 1  - \varepsilon_d(m)^2  + (\underline \gamma^c(m) \Vert\eta_d\Vert^2)^{-1}(\widehat s^2 / s^2)
                             \underline \gamma^c(m)\, \nu\! <\!  [b_m]^{-1} , [G_m]\! >_F $.
Since $[G_m] = \sum_{\alpha=1}^N \mu_\alpha \, \bfvarphi^\alpha \otimes \bfvarphi^\alpha$ and using the proof of Lemma~\ref{lemma:8.10}, it can be deduced that
$<\!  [b_m]^{-1}\, , [G_m]\! >_F = \sum_{\alpha=1}^N  \sum_{\beta=1}^N (\mu_\beta/\zeta_\alpha(m)) <\! \bfvarphi^\alpha \otimes \bfvarphi^\alpha\, , \bfvarphi^\beta \otimes \bfvarphi^\beta\! >_F = \sum_{\alpha=1}^m 1/\zeta_\alpha(m) = m\,\sigma^2 = m \,\Vert\eta_d\Vert^2/(\nu N)$. Therefore,
$r^c(m)\, \underline g(m) =  1  - \varepsilon_d(m)^2 + (\widehat s^2 / s^2) (m/N)$.
It can be seen that the right-hand side of this equation is $\underline g(m)$ defined by equation \eqref{eq:7.12}. Consequently, $r^c(m)=1$.
\end{proof}
\begin{remark}[\textbf{MaxEnt approximation} $d_N^{2,c}(m)$ \textbf{of} $d_N^2(m)$ \textbf{for all} $m\geq m_\opt$]\label{remark:8.12}
Using Theorem~\ref{theorem:7.13} and equation \eqref{eq:7.7a}, the MaxEnt approximation of $d_N^2(m)$ is defined, for all $m\geq m_\opt$, as $d_N^{2,c}(m) = f_d(m) + r^c(m)\, \underline g(m)$ in which $f_d(m)$ is defined by equation \eqref{eq:7.7b} and $\underline g(m)$ by the first equation \eqref{eq:7.12}. From equations \eqref{eq:8.10}, Theorem~\ref{theorem:7.13} and its proof,  it can be deduced that for all $m\geq m_\opt$,  $d_N^{2,c}(m) = f_d(m) +\underline g(m) = d_N^{2,\psup}(m)$, and consequently, using equation \eqref{eq:7.18},
\begin{equation}\label{eq:8.11}
\forall m\geq m_\opt \,\, , \, \, d_N^{2,c}(m) = 1 +\frac{m}{N-1}   \quad ; \quad d_N^{2,c}(N) = d_N^2(N) = 1 + \frac{N}{N-1} \, .
\end{equation}
\end{remark}
\begin{remark}[\textbf{Rough approximation} $d_N^{2,\papprox}(m)$ \textbf{of} $d_N^2(m)$ \textbf{for all} $m\geq m_\opt$]\label{remark:8.13}
In this remark, for $m\geq m_\opt$, we define a "rough approximation" $r_\papprox(m)$ of $r(m)$ defined by equation \eqref{eq:7.11}.
Let $\bfj_o =(1,2,\ldots , N)\in \curJ\subset \NN^N$.
Equation \eqref{eq:4.3} shows that $[\eta_d(\bfj_o)] = [\eta_d]$, and consequently, equations \eqref{eq:7.4} and \eqref{eq:7.9} yield $[\eta_d^m(\bfj_o)] = [\eta^m_d]$. Hence, equation \eqref{eq:7.8} yields
$\gamma_{\bfj_o}(m) = \exp(-(2s^2)^{-1} \Vert \eta_d - \eta_d^m\Vert^2)$, which can be rewritten, using equation \eqref{eq:7.5}, as
$\gamma_{\bfj_o}(m) = \exp(-(2s^2)^{-1} \varepsilon_d(m)^2\Vert \eta_d\Vert^2)$. Let us assume that, for $m\geq m_\opt$, $\underline\gamma(m) \simeq \underline\gamma(N) = 1$ (due to equation \eqref{eq:7.13}).
Starting from equation \eqref{eq:7.11}, we define  $r_\papprox(m) = N^{-N} (\sum_{\bfj\in\curJ} \gamma_{\bfj_o}(m)\, g_\bfj(m))\, (\underline\gamma(N)\, \underline g(m))^{-1}$ $= \gamma_{\bfj_o}(m)$.
Therefore, $\forall m\geq m_\opt$, $r_\papprox(m) = \exp(- (2s^2)^{-1} \, \varepsilon_d(m)^2 \, \Vert\eta_d\Vert^2)$.
Since $\varepsilon_d(N) = 0$ (see Lemma~\ref{lemma:7.7}) and since $\gamma_{\bfj_o}(N) = 1$, it can be seen that $r_\papprox(N)=r(N)=1$. Finally, using equations \eqref{eq:7.7a} and \eqref{eq:7.14}, the corresponding approximation $d_N^{2,\papprox}(m)$ of $d_N^2(m)$ is written as
\begin{equation}\label{eq:8.12}
\forall m\geq m_\opt \quad , \quad  d_N^{2,\papprox}(m) = f_d(m) + \underline g(m)\, \exp(-\frac{1}{2s^2} \, \varepsilon_d(m)^2 \, \Vert\eta_d\Vert^2) \, ,
\end{equation}
in which $f_d(m)$ is defined by equation \eqref{eq:7.7b}, $\underline g(m)$ by the first equation \eqref{eq:7.12}, and $\varepsilon_d(m)$ by equation \eqref{eq:7.5}.
Since $r_\papprox(N)=1$, using Lemma~\ref{lemma:7.12}, the third equation \eqref{eq:7.12}, and equation \eqref{eq:3.2} yield
$d_N^{2,\papprox}(N) = d_N^2(N) = 1 + N/(N-1)$.
\end{remark}
\section{Numerical illustration}
\label{sec:9}
The numerical illustration proposed is the application (AP1) in Section~10 of reference \cite{Soize2019d}.
For reasons of limitation of the paper length, we cannot reproduce the description of this application and we refer the reader to the given reference.
With respect to the notation introduced in \eqref{sec:1.1}, we have $n_w=20$, $n_q=200$, $n=220$, and $N=200$.
For the PCA (see \eqref{sec:2}) and for  $\varepsilon = 10^{-6}$ in the second equation \eqref{eq:2.1}, we have $\nu=9$. Consequently, $\err_\PCA(\nu) \leq 10^{-6}$.
Concerning the nonparametric estimate (see \eqref{sec:3}), the values of the parameters defined by equation \eqref{eq:3.2} are
$s=0.615$, $\widehat s = 0.525$, and $\widehat s / s = 0.853$.
The use of equations \eqref{eq:5.1} and \eqref{eq:5.2} yields $\varepsilon_\opt = 60$ and $m_\opt=10$. Parameter $\kappa$ has been fixed to $1$.
The graph of function $\alpha \mapsto \log(\lambda_\alpha(\varepsilon_\opt))$ (see \eqref{sec:5.1}) is displayed in \eqref{figure:2} (left) and the graph of function $m\mapsto \varepsilon_d(m)$  defined by equation \eqref{eq:7.5} is shown in \eqref{figure:2} (right). In order to better visualize these graphs, a zoom has been done for the abscissa ($\alpha\leq 50$ and $m\leq 50$ instead of the upper bound $N = 200$). It can be seen that these graphs are similar to the ones shown in \eqref{figure:1} and that Hypotheses~\ref{hypothesis:5.2} and \ref{hypothesis:7.8} are well verified.
\begin{figure}[htbp]
  \centering
  \includegraphics[width=6.0cm]{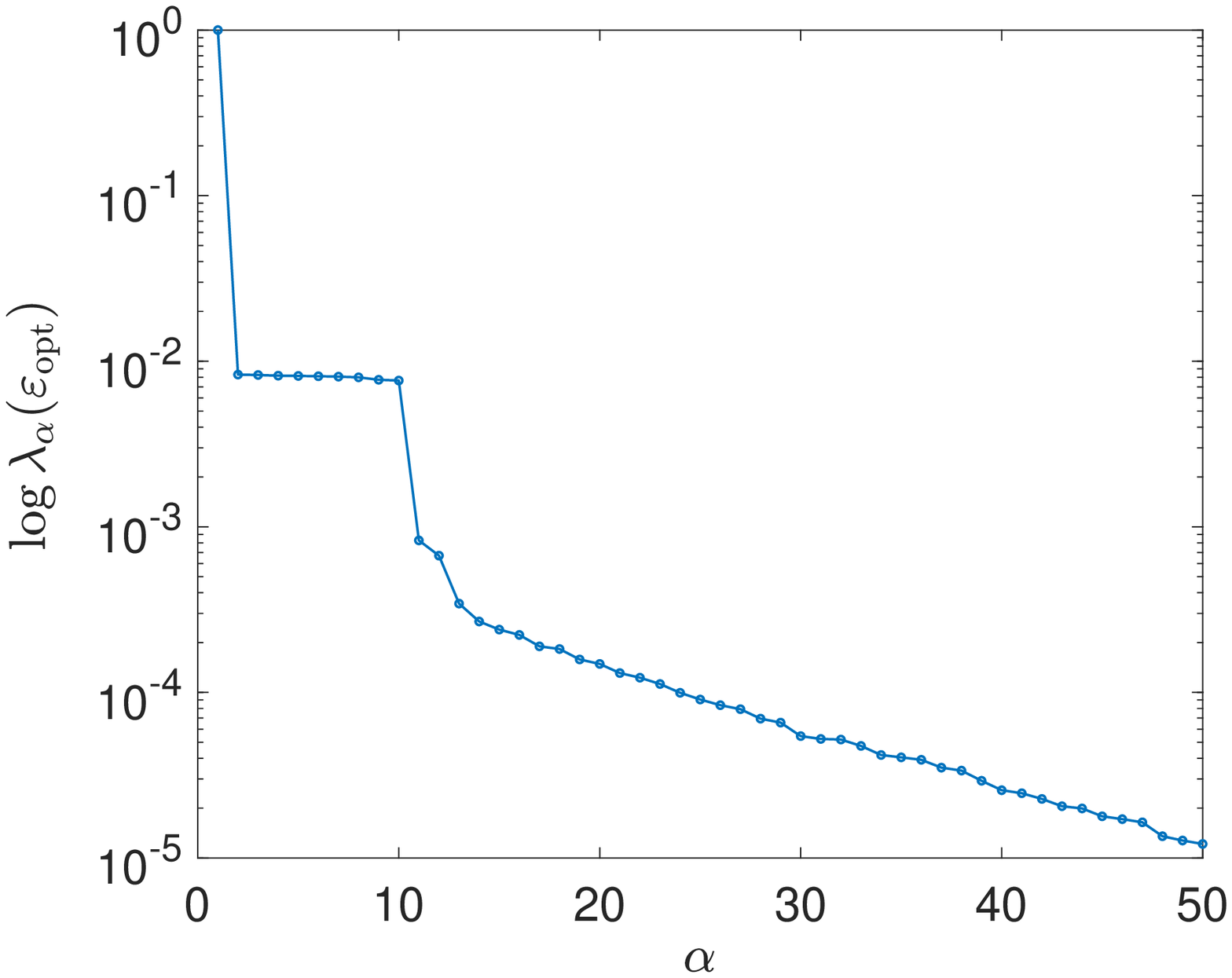}  \includegraphics[width=6.0cm]{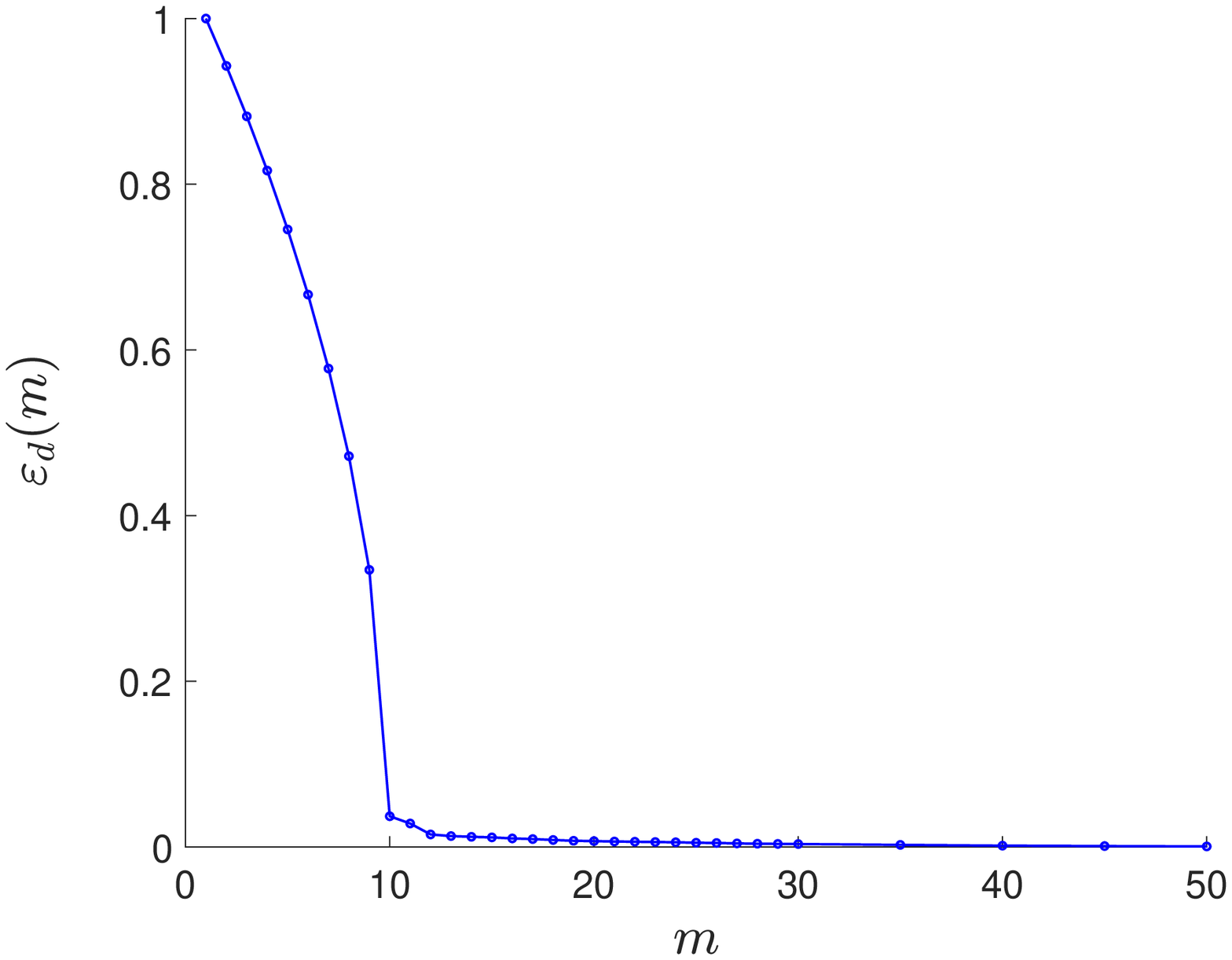}
  \caption{Left figure: distribution of the eigenvalues $\lambda_\alpha(\varepsilon_\opt)$ in log scale as a function of rank $\alpha\leq 50$ for $\varepsilon_\DM = \varepsilon_\opt = 60$. Right figure: graph of function $m\mapsto \varepsilon_d(m)$ for $m\leq 50$.}
  \label{figure:2}
\end{figure}
The graph of function $m \mapsto f_d(m)$ defined by equation \eqref{eq:7.7b} is displayed in \eqref{figure:3} (left) and the graph of function $m\mapsto \underline g(m)$ defined by equation \eqref{eq:7.12} is shown in \eqref{figure:2} (right). It can be seen that $f_d$ has a minimum for $m=m_\opt = 10$.
\begin{figure}[htbp]
  \centering
  \includegraphics[width=6.0cm]{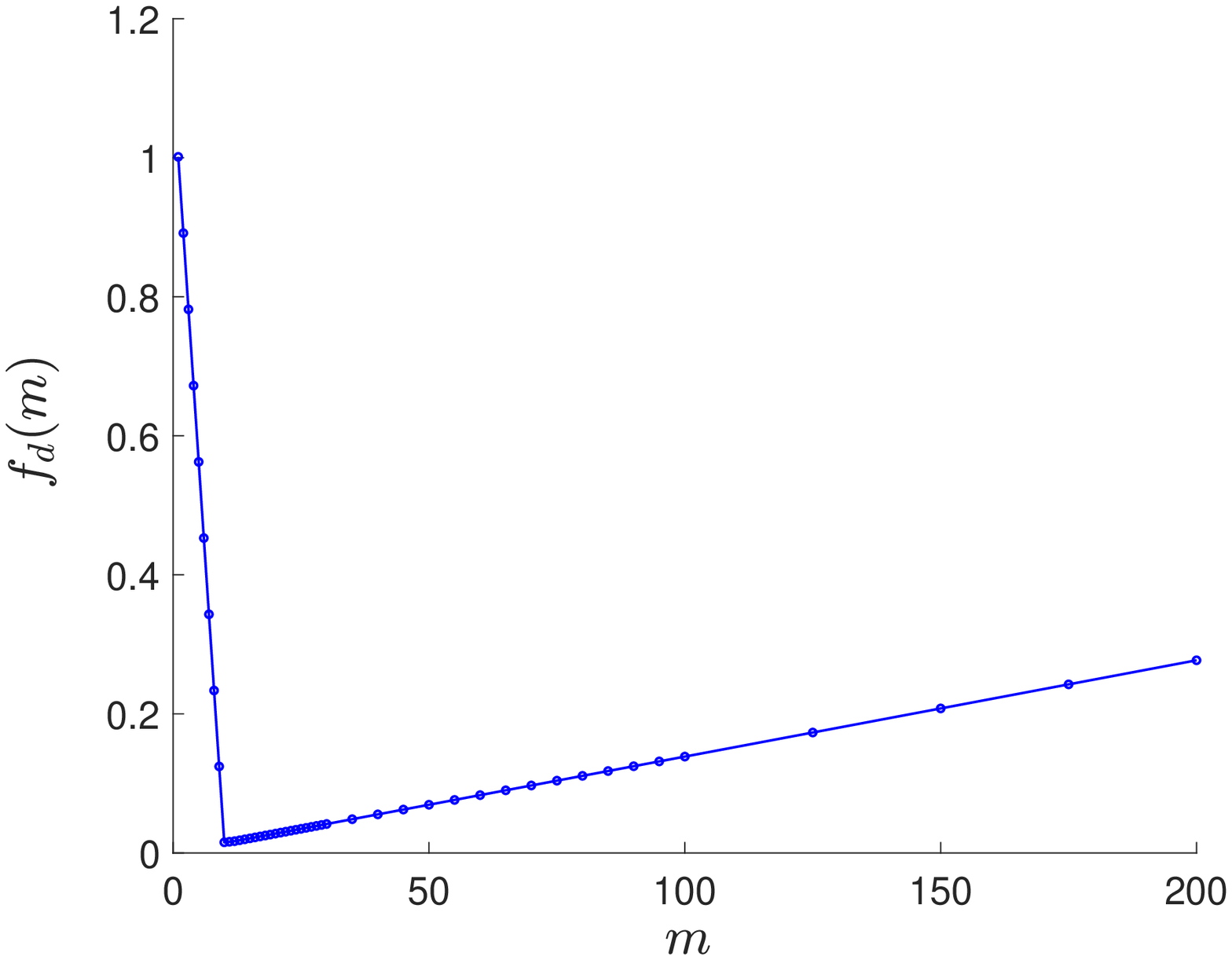}  \includegraphics[width=6.0cm]{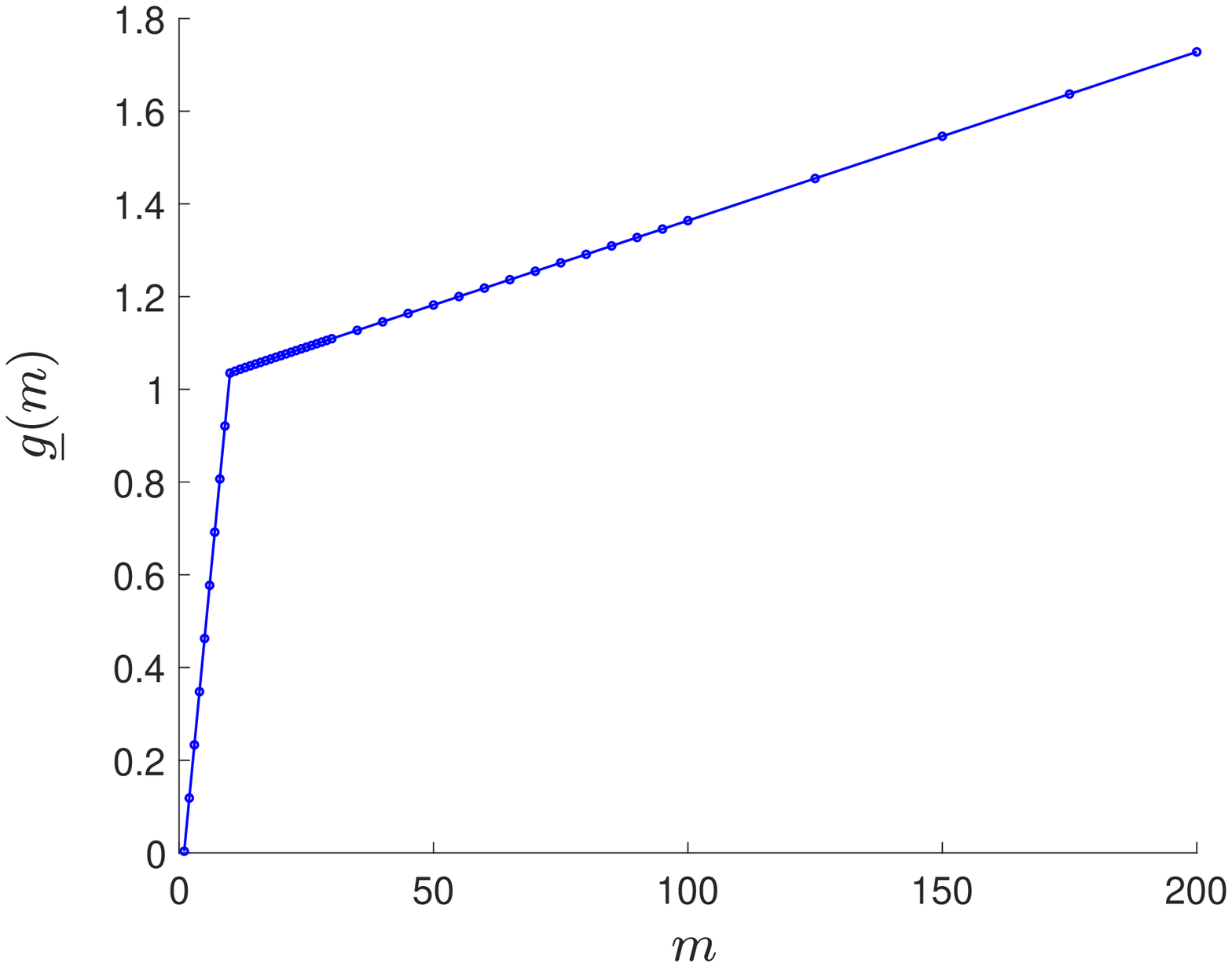}
  \caption{Left figure: graph of function $m \mapsto f_d(m)$. Right figure: graph of function $m\mapsto \underline g(m)$.}
  \label{figure:3}
\end{figure}
For all $m$ such that $1\leq m\leq N$, the estimation $d_N^{2,\simul}(m) =  \Vert\eta_d\Vert^{-2} \frac{1}{n_\pMC} \sum_{\ell=1}^{n_\pMC} \Vert \, [\eta_\ar^\ell] - [\eta_d]\Vert^2$ of $d_N^2(m)$ defined in Definition~\ref{definition:7.1} has been carried out using the learned dataset $\{[\eta_\ar^\ell]\, , \ell =1,\ldots , n_\pMC\}$ with $n_\pMC =320\, 000$ realizations of random matrix $[\bfH^N_m]$, which  have been computed with the PLoM method presented in \eqref{sec:6}. It has been verified that the $L^2$-convergence is obtained for this value of $n_\pMC$. Left \eqref{figure:4} shows the graph of function $m\mapsto d_N^{2,\simul}(m)$. It can be seen that the local minimum is a global minimum obtained for $ m = m_\opt$ as expected and that $d_N^{2,\simul}(N) \simeq 2$ (in agreement with Lemma~\ref{lemma:7.2}). Right \eqref{figure:4} shows three curves: again the graph of $m\mapsto d_N^{2,\simul}(m)$ in order to have a reference, and for $m\geq m_\opt$, the graph of function $m \mapsto d_N^{2,c}(m)$ computed with equation \eqref{eq:8.11} and the graph of function $m \mapsto d_N^{2,\papprox}(m)$ computed with equation \eqref{eq:8.12}. It can be seen that the graph of $m \mapsto d_N^{2,c}(m)$ is in coherence with Theorem~\ref{theorem:7.13} and that
the graph of function $m \mapsto d_N^{2,\papprox}(m)$ has a minimum in $ m = m_\opt = 10$ on $\curM_\opt$, as expected.
\begin{figure}[htbp]
  \centering
  \includegraphics[width=6.0cm]{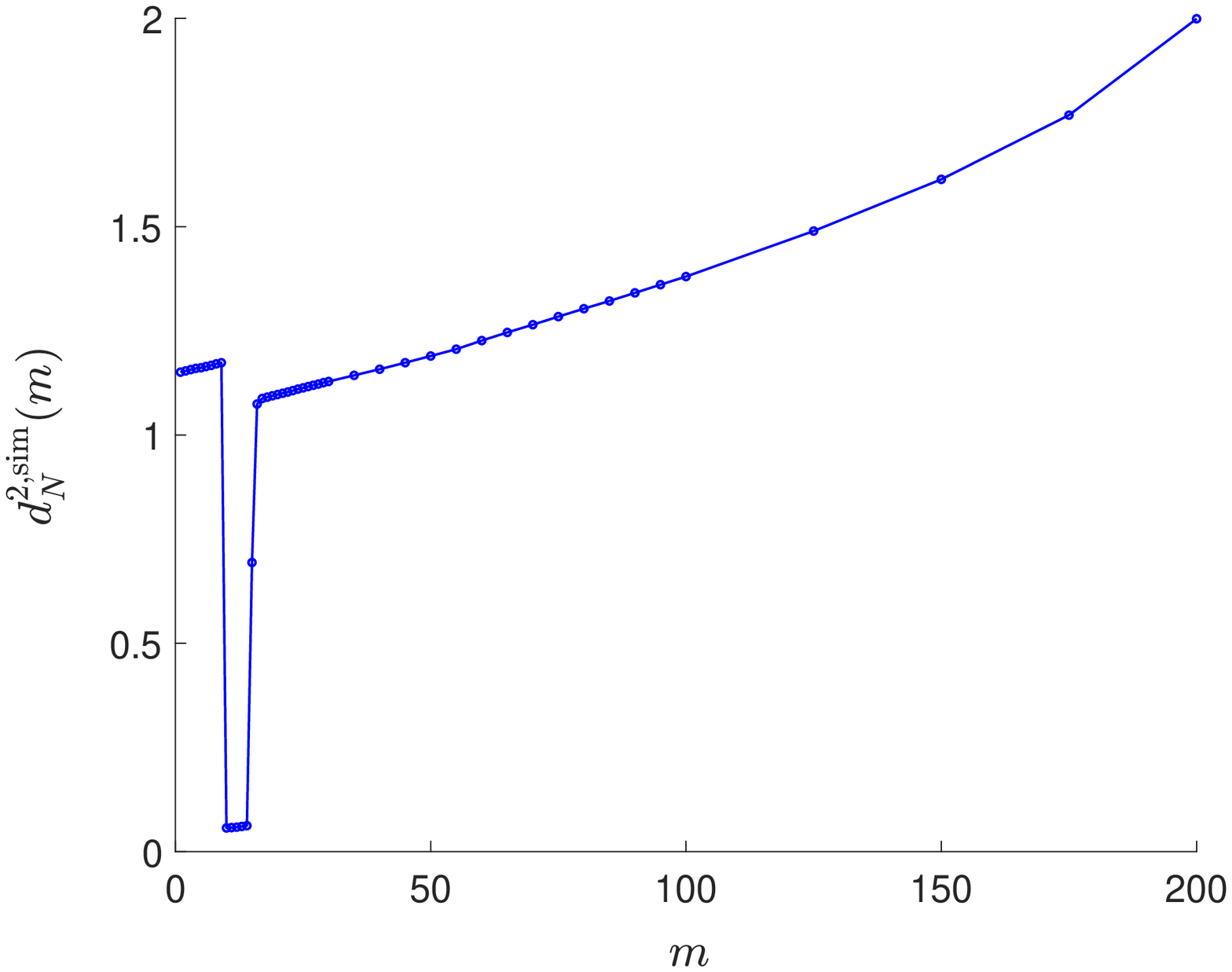}  \includegraphics[width=6.0cm]{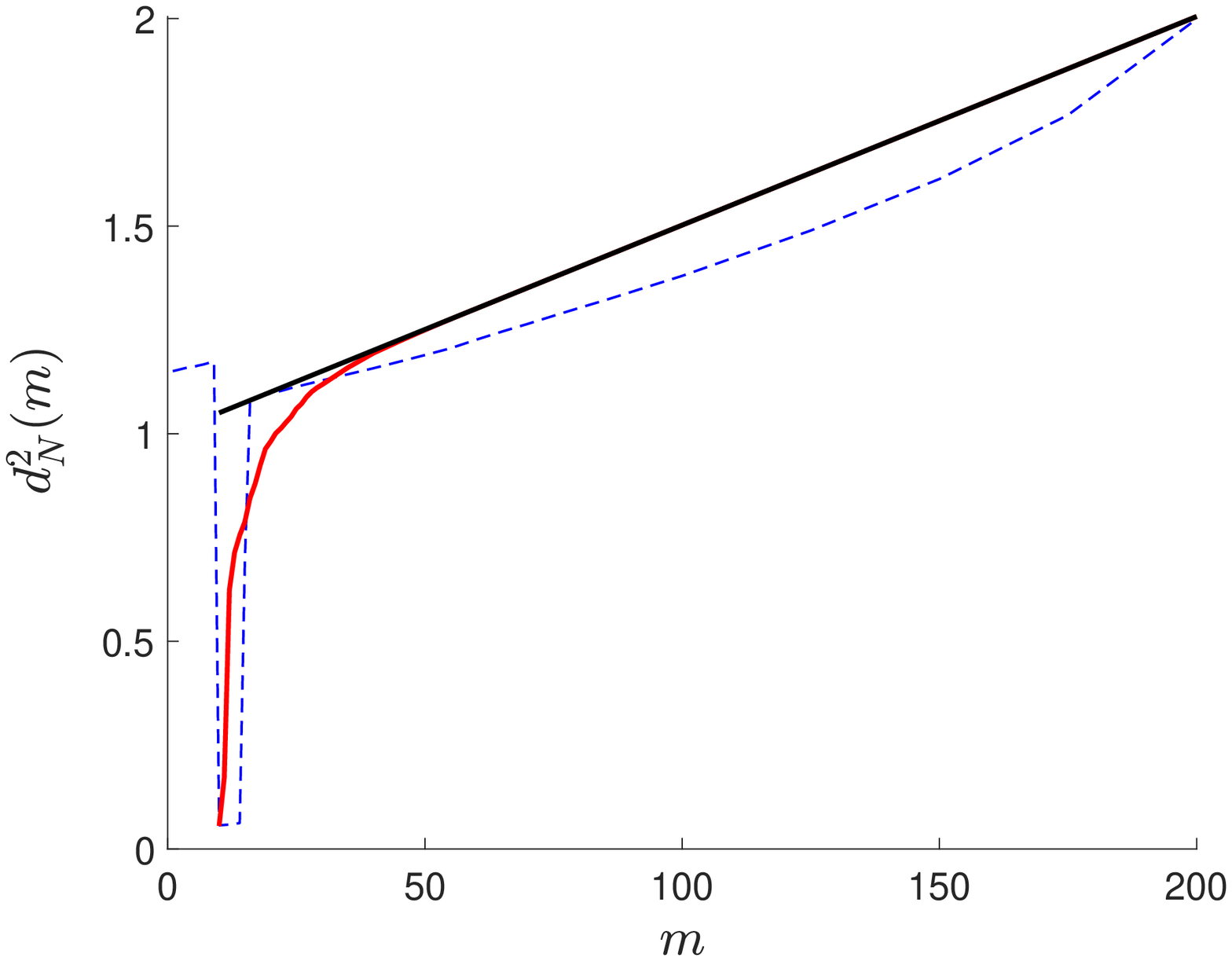}
  \caption{Left figure: graph of function $m \mapsto d_N^{2,\simul}(m)$. Right figure: graph of function $m \mapsto d_N^{2,\simul}(m)$ (blue dashed line), and for $m\geq m_\opt$, graphs of $m \mapsto d_N^{2,c}(m)$ (dark thick straight line) and $m \mapsto d_N^{2,\papprox}(m)$ (red thick curve line).}
  \label{figure:4}
\end{figure}
\section{Conclusions}
\label{sec:10}
In this paper, we have presented mathematical results that justify, highlight, and better explain the probabilistic learning on manifolds proposed in \cite{Soize2016}. We have formulated and proven several results, which show that the PLoM methodology is efficient for probabilistic learning as it has been demonstrated  in the framework of applications performed for complex engineering systems.  The distance introduced for the mathematical analysis of the concentration properties of the probability measure could be used to estimate the optimal dimension of the reduced-order diffusion maps basis and thus to replace the algorithm previously introduced, which uses only the initial dataset. However, the criterion based on this distance would require to generate a large number of replicates of the learned dataset and therefore would induce a larger numerical cost.
\section*{Acknowledgments}
This research was partially supported by the PIRATE pro\-ject funded under DARPA’s AIRA program and by FASTMATH SciDac Institute supported under DOE's ASCR program.
\bibliographystyle{elsarticle-num}
\bibliography{references}
\end{document}